\documentclass[a4paper,10pt]{amsart}
\usepackage[utf8x]{inputenc}
\usepackage{amssymb}
\usepackage{amsmath}
\usepackage{amsthm,mathbbol}
\usepackage{mathrsfs}
\usepackage{bbm}
\usepackage[all]{xy}

\usepackage{IEEEtrantools}

\usepackage{stmaryrd}

\usepackage{tikz}
\usetikzlibrary{matrix,arrows,decorations.pathmorphing}
\usetikzlibrary{positioning}
\usepackage[parfill]{parskip}
\usepackage[active]{srcltx}

\title{Simple transitive $2$-representations of Soergel bimodules in type $B_2$}
\author{Jakob Zimmermann}
\date{\today}

\newtheorem{theorem}{Theorem}[section]
\newtheorem{proposition}[theorem]{Proposition}
\newtheorem{lemma}[theorem]{Lemma}
\newtheorem{corollary}[theorem]{Corollary}
\theoremstyle{definition}
\newtheorem{example}[theorem]{Example}

\newtheorem{remark}[theorem]{Remark}

\newcommand{\setof}[2]{\{#1\;|\; #2\}}

\newcommand{\C}{\mathbb{C}}

\newcommand{\R}{\mathbb{R}}
\newcommand{\Z}{\mathbb{Z}}

\renewcommand{\l}{\mathcal{L}}

\renewcommand{\j}{\mathcal{J}}

\renewcommand{\c}{\clubsuit}
\usepackage{enumerate}

\begin{document}

\begin{abstract}
We prove that every simple transitive $2$-representation of the
fiat $2$-category of Soergel bimodules (over the coinvariant algebra)
in type $B_2$ is equivalent to a cell $2$-representation. 
We also describe some general properties of the 
$2$-category of Soergel bimodules for arbitrary finite dihedral groups.
\end{abstract}

\maketitle

\section{Introduction}
\noindent

Understanding the approach of categorification via $2$-representation theory
of $2$-categories has its roots in the papers
\cite{BernsteinFrenkelKhovanov99,ChuangRouquier08}. 
In 2010, Mazorchuk and Miemietz started the series  
\cite{MazorchukMiemietz1,MazorchukMiemietz2,MazorchukMiemietz3,MazorchukMiemietz4,MazorchukMiemietz5,MazorchukMiemietz6} 
of papers in which they systematically study the $2$-representation theory of finitary 
and fiat  $2$-categories. 
The latter  can be thought of as analogues of finite dimensional algebras 
(for finitary $2$-categories) or finite dimensional
algebras with involution (for fiat $2$-categories). The first paper 
\cite{MazorchukMiemietz1} in the series introduces the notion 
of {\em cell $2$-representations} as a possible candidate for the notion of 
``simple'' $2$-representations in this setting.

Cell $2$-representations are inspired by Kazhdan-Lusztig cell modules for Hecke algebras, as
defined in \cite{KahzdanLusztig79}. In the case of $2$-categories, a left (right, two-sided)
cell is a subset of indecomposable $1$-morphisms which generate the same 
left (resp. right, two-sided) ideal. Given a left cell,  the corresponding
cell $2$-representation is a suitable subquotient of the
principal $2$-representation associated with this left cell. 
In \cite{MazorchukMiemietz1} it was shown that cell $2$-representations
have a lot of properties similar to properties of usual simple representations.

The paper \cite{MazorchukMiemietz5} introduces the notion of {\em simple transitive $2$-representations}. 
By a \emph{transitive} $2$-representation one means an additive 
$2$-representation for which the action of $1$-morphisms is transitive in the sense that, 
starting from any  indecomposable object and applying all $1$-morphisms, one
obtains the whole underlying category of the $2$-representation by taking the additive closure.
A transitive $2$-representation is called \emph{simple transitive} if, in addition to the above defined notion of 
transitivity, we have that the maximal ideal of the $2$-representation 
which is invariant under the $2$-action is zero. In other words, the notion of 
simple transitivity has  two layers, were the first layer addresses the level of $1$-morphisms
and the second  layer addresses  the level of $2$-morphisms.

By construction, all cell $2$-representations are simple transitive but it is not obvious 
whether each simple transitive $2$-representation is equivalent to a cell
$2$-representation. As it turns out, in general, these two classes of $2$-representations
are different. Indeed, in \cite{MazorchukMiemietz5} one finds an example of a $2$-category for which 
there exist simple transitive $2$-representations which are not cell $2$-representations. 
One of the main results of \cite{MazorchukMiemietz5} is that, if the cell structure of the 
underlying $2$-category is ``nice enough'', then these two classes of $2$-representations coincide. 

One of the main examples for such a ``nice'' $2$-category is the $2$-category of Soergel 
bimodules over the coinvariant algebra of type $A$. Hence, for this $2$-category,
each simple transitive $2$-representation is equivalent to a cell $2$-representation.
Moreover, from \cite{MazorchukMiemietz1} it is also known that, for Soergel bimodules in type $A$,
two cell $2$-representations corresponding to left cells inside the same two-sided cell
are equivalent.  For all other types, the cell structure is more complicated and no
classification of simple transitive $2$-representations is known. 

In the present paper we study simple transitive $2$-representations of the 
the $2$-category of Soergel bimodules over the coinvariant algebra of type $B_2$.
This is the smallest case for which the cell structure does not satisfy 
the requirements of \cite[Theorem 18]{MazorchukMiemietz5}.
Our main result is the following (see Theorem~\ref{MainThm}):
\vspace{2mm}

{\bf Theorem~A.}
{\it
Each simple transitive $2$-representation of the $2$-category of Soergel bimodules over 
the coinvariant algebra of type $B_2$ is equivalent to a cell $2$-representation.
}
\vspace{2mm}

However, in contrast to type $A$, we will show that, in type $B_2$, the cell 
$2$-rep\-re\-sen\-ta\-ti\-ons corresponding to the two different left cells inside the
unique non-singleton two-sided cell are not equivalent. Behind this phenomenon
is the fact that the Kazhdan-Lusztig cell representations to which these different
cell $2$-representations decategorify are not isomorphic. The latter is due to the fact that 
these Kazhdan-Lusztig cell representations contain non-isomorphic one-dimensional subquotients. 
Additionally to the main result formulated above, we study the
$2$-category $\mathscr{S}_n$ of Soergel bimodules for an arbitrary finite dihedral group $D_n$ 
and classify for this $2$-category all simple transitive $2$-representations
of small ranks (namely, ranks one and two), generalizing \cite[Proposition~21]{MazorchukMiemietz5}.
Here, the {\em rank} of an additive  $2$-representation is the number of 
isomorphism classes of indecomposable additive generators of the underlying category of
the $2$-representation.

The proof of the main theorem, which will be given in Section~\ref{D4Case}, 
can be divided in two parts. In the first part we study the decategorification of 
a given simple transitive $2$-representation of $\mathscr{S}_4$ and show that 
all simple transitive $2$-representations have either rank one or rank three.
In this part we benefit from the fact that the representation theory of the dihedral group 
is well-known and quite easily described. Another crucial result that we use is the classical 
Perron-Frobenius Theorem describing the structure of real matrices with positive
or non-negative entries. 

What is left to show then is that every simple transitive $2$-representation of 
$\mathscr{S}_4$ of rank one or three is equivalent to a cell $2$-representation. 
The rank one case is a bit easier and is treated in Section~\ref{trivialCellModules}. 
The rank three case, on the other hand, is more involved and constitutes 
the second part of the proof of the main theorem. It is proved by giving an 
explicit construction of an equivalence. 

The article is organized as follows. In the next section we collect all preliminaries 
about $2$-categories, define fiat and finitary $2$-categories and their decategorifications. 
Section~\ref{scombin} describes the combinatorics of $2$-categories; more precisely, the notion 
of cells of $2$-categories is defined and cell $2$-representations and simple transitive 
$2$-representations are introduced. In Section~\ref{sSoergel}, we define Soergel bimodules 
and describe the cell structure of the $2$-category of Soergel bimodules for the dihedral
group $D_n$. In particular, these introductory sections summarize definitions and 
necessary results from  \cite{MazorchukMiemietz1,MazorchukMiemietz2,MazorchukMiemietz3,
MazorchukMiemietz4,MazorchukMiemietz5,MazorchukMiemietz6}.
Section \ref{ssimple} collects some preliminary results on 
simple transitive $2$-representations of $\mathscr{S}_n$. In Section~\ref{sd4} we classify all 
simple transitive $2$-representations of $\mathscr{S}_4$. Finally, in Section 
\ref{Rank1and2Case} we examine the situation for $n \geq 5$ and prove that every simple transitive 
$2$-representation of rank one is equivalent to a cell $2$-representation, and that there are no 
simple transitive $2$-representations of rank two. 

\section{Preliminaries}
\label{sprelim}

\subsection{Notation}

We fix an algebraically closed field  $\Bbbk$. If not stated otherwise, all tensor 
products are over $\Bbbk$ and all categories are assumed to be $\Bbbk$-linear.

\subsection{$2$-categories}

A \emph{$2$-category} is a category enriched over the category of small categories. 
A $2$-category $\mathscr{C}$ consists of objects, denoted $\mathtt{i, j}, \dots$; 
$1$-morphisms, denoted $F, G, \dots$; and $2$-morphisms, denoted $\alpha, \beta, \dots$. 
Each $\mathscr{C}(\mathtt{i},\mathtt{j})$ is a small category and composition is 
bifunctorial. For every $\mathtt{i} \in \mathscr{C}$, the corresponding identity 
$1$-morphism is $\mathbbm{1}_{\mathtt{i}}$. For every $1$-morphism $F$, 
the corresponding identity $2$-mor\-phism is $\text{id}_F$. Horizontal composition of 
$1$-morphisms is denoted by $\circ$. Horizontal and vertical compositions of $2$-morphisms 
are denoted by $\circ_0$ and $\circ_1$, respectively. 

\subsection{Finitary $2$-categories} 

An additive, $\Bbbk$-linear, idempotent split category $\mathcal{C}$ is called \emph{finitary} 
if $\mathcal{C}$ has finitely many isomorphism classes of indecomposable objects and all 
morphism spaces are finite dimensional $\Bbbk$-vector spaces. We denote by $\mathfrak{A}_{\Bbbk}$ 
the $2$-category whose objects are finitary $\Bbbk$-linear categories, whose $1$-morphisms are 
additive $\Bbbk$-linear functors and whose $2$-morphisms are natural transformations. 
From now on we assume that $\mathscr{C}$ is a \emph{finitary} 
$2$-category, that is:
\begin{itemize}
\item $\mathscr{C}$ has a finite number of objects;
\item $\mathscr{C}(\mathtt{i},\mathtt{j})\in \mathfrak{A}_{\Bbbk}$,
for all $\mathtt{i,j}$, and horizontal composition is additive
 and $\Bbbk$-linear;
\item for any object $\mathtt{i}$ in $\mathscr{C}$, the $1$-morphism $\mathbbm{1}_{\mathtt{i}}$ 
 is indecomposable.
\end{itemize} 
For examples and more details, see \cite{MazorchukMiemietz1,MazorchukMiemietz2,MazorchukMiemietz6,GM1,GM2,Xantcha,Zhang,Zhang2}.

\subsection{$2$-representations}

A \emph{$2$-representation} of $\mathscr{C}$ is a strict $2$-functor from $\mathscr{C}$ 
to \textbf{Cat}, the $2$-category of small categories.  The $2$-category $\mathscr{C}$-afmod of 
\emph{finitary} $2$-rep\-re\-sen\-tations has strict $2$-functors from $\mathscr{C}$ 
to $\mathfrak{A}_{\Bbbk}$ as objects; $2$-natural transformations as $1$-morphisms, and 
modifications as $2$-morphisms, see \cite[Section 2.3]{MazorchukMiemietz3}. We denote 
$2$-representations $\textbf{M, N, }\dots$. We say that $\mathbf{M}$ and $\mathbf{N}$ 
are \emph{equivalent} if there is a $2$-natural transformation 
$\Phi: \textbf{M} \to \textbf{N}$ such that each $\Phi_{\mathtt{i}}$ is an 
equivalence.

Let $\mathbf{M}$ be a $2$-representation of $\mathscr{C}$ such that $\mathbf{M}(\mathtt{i})$ 
is idempotent split and additive, for each $\mathtt{i} \in \mathscr{C}$. Let $\{X_i:i\in I\}$
be a collection of objects in (various) $\mathbf{M}(\mathtt{j})$. The additive closure 
$\text{add}(\{\mathbf{M}(F)X_i\})$ of all objects of the form 
$\mathbf{M}(F)X_i$, where $F$ runs through all $1$-morphisms in $\mathscr{C}$ and $i\in I$, 
is $\mathscr{C}$-stable and gives, by restriction, a $2$-representation $\mathscr{C}$
denoted by $\textbf{G}_{\textbf{M}}(\setof{X_i}{i \in I})$. 

For simplicity, we will often write $F\:X$ instead of $\mathbf{M}(F)\:X$, etc.  
For $\mathtt{i} \in \mathscr{C}$,  let $\mathbf{P}_{\mathtt{i}}:=\mathscr{C}(\mathtt{i},\_)
\in \mathscr{C}$-afmod be the $\mathtt{i}$\emph{-th principal additive} 
$2$-representation.

\subsection{Weakly fiat and fiat $2$-categories} 

For a $2$-category $\mathscr{C}$, we denote by $\mathscr{C}^{\text{op}}$ the opposite $2$-category
which we obtain by reversing both $1$- and $2$-morphisms. A finitary $2$-category $\mathscr{C}$ 
is called \emph{weakly fiat} if 
\begin{itemize}
\item there is a weak equivalence $*: \mathscr{C} \to \mathscr{C}^{\text{op}}$;
\item for all $\mathtt{i,j} \in \mathscr{C}$ and $F \in \mathscr{C}(\mathtt{i},\mathtt{j})$,
there are $\alpha: F \circ F^* \to \mathbbm{1}_{\mathtt{j}}$  and 
$\beta: \mathbbm{1}_{\mathtt{i}} \to F^* \circ F$ such that 
$(\alpha \circ_0 \text{id}_F) \circ_1 (\text{id}_F \circ_0 \beta) = \text{id}_F$ and 
$(\text{id}_{F^*}\circ_0 \alpha) \circ_1 (\beta \circ_0 \text{id}_{F^*}) = \text{id}_{F^*}$.
\end{itemize}
A weakly fiat $2$-category is \emph{fiat} if $*$ is involutive, see \cite[Section 2.4]{MazorchukMiemietz1} 
and \cite[Section 2.2]{MazorchukMiemietz2}.

\begin{example}\label{CA}
Let $A$ be a basic, self-injective, weakly symmetric, connected and not simple $\Bbbk$-algebra of 
dimension $m<\infty$ and $\mathcal{A}$ a small category equivalent to $A$-mod.
Following  \cite[Section~5]{MazorchukMiemietz5}, we define
the fiat $2$-category $\mathscr{C}_A$ as follows: 
\begin{itemize}
\item $\mathscr{C}_A$ has one object $\c$ (which we identify with $\mathcal{A}$);
\item $1$-morphisms are direct sums of functors isomorphic to the identity functor 
on $\mathcal{A}$ or to a functor given by tensoring with a projective $A$-$A$-bimodule;
\item $2$-morphisms are natural transformations.
\end{itemize} 
\end{example}

For more examples, see \cite{MazorchukMiemietz1,MazorchukMiemietz6,Xantcha}.

\subsection{Abelianization}

For a finitary category $\mathcal{A}$, we denote by $\overline{\mathcal{A}}$ its
\emph{abelianization} as in \cite[Subsection~2.7]{MazorchukMiemietz5}.
This induces the \emph{abelianization $2$-functor} $\mathbf{M}\mapsto \overline{\mathbf{M}}$,
see \cite[Subsection~4.2]{MazorchukMiemietz2}, and gives  the 
$\mathtt{i}$\emph{-th principal abelian} $2$-representation $\overline{\mathbf{P}}_{\mathtt{i}}$. 

\subsection{Decategorification}

The \emph{Grothendieck group} of a skeletally small abelian category $\mathcal{A}$ is denoted
$[\mathcal{A}]$. Similarly, the \emph{split Grothendieck group} of a skeletally small additive 
category is denoted $[\mathcal{A}]_{\oplus}$. We set 
$[\mathcal{A}]^{\C} := \C \otimes_{\Z} [\mathcal{A}]$ and similarly for 
$[\mathcal{A}]_{\oplus}$. If $\mathcal{A}$ is additive but not abelian, we often simplify
$[\mathcal{A}]_{\oplus}$ to $[\mathcal{A}]$.

The \emph{decategorification} $[\mathscr{C}]$ of $\mathscr{C}$ is a ($1$-)category with 
the same objects as $\mathscr{C}$ and $[\mathscr{C}](\mathtt{i,j}):=[\mathscr{C}(\mathtt{i,j})]$,
for all $\mathtt{i,j}$. Composition in $[\mathscr{C}]$ is induced by that in $\mathscr{C}$.
Given a $2$-representation $\mathbf{M}$ of $\mathscr{C}$, the \emph{decategorification} 
$[\mathbf{M}]$ of $\mathbf{M}$ is the functor from $[\mathscr{C}]\to \mathbf{Ab}$
(the category of abelian groups) defined as follows: 
\begin{itemize}
\item for $\mathtt{i} \in [\mathscr{C}]$, $[\mathbf{M}](\mathtt{i})=[\mathbf{M}(\mathtt{i})]$;
\item for $F \in \mathscr{C}(\mathtt{i,j})$, the action $[F]\curvearrowright[\mathbf{M}](\mathtt{i})$ 
is induced from $F\curvearrowright\mathbf{M}(\mathtt{i})$.
\end{itemize}

\begin{example}\label{decatPrincip}
Let $\mathscr{C}$ be a finitary $2$-category with one object denoted by $\c$. Then 
$[\mathbf{P}_{\c}]$ can be identified with the regular representation of the ring 
$[\mathscr{C}(\c, \c)]$. 
\end{example}

\subsection{Matrices in the Grothendieck group}\label{MatInGroGr}

Let $\mathbf{M}$ be a finitary $2$-representation of a weakly fiat $2$-category $\mathscr{C}$ with only 
one object $\c$. Then, for a $1$-morphism $F$, we denote by $\Lparen F\Rparen$ the square matrix 
with non-negative integer coefficients, whose rows and columns are indexed by isomorphism classes 
of indecomposable objects in $\mathbf{M}(\clubsuit)$, and where the intersection of the row indexed by 
$Y$ where the column indexed by $X$ contains the multiplicity of $Y$ as a direct summand of $F\:X$.  

Similarly, for the abelianization $\overline{\mathbf{M}}$ of $\mathbf{M}$, we  denote by 
$\llbracket F \rrbracket$ the matrix  with rows and columns indexed by isomorphism classes of 
simple objects in $\overline{\mathbf{M}}(\clubsuit)$ and where the intersection of the row 
indexed by $Y$ and the column indexed by $X$ contains the composition multiplicity of $Y$ in $F\:X$.
We have $\llbracket F^* \rrbracket = \Lparen F\Rparen^t$ by  \cite[Lemma~10]{MazorchukMiemietz5}.

\section{Combinatorics of $2$-categories and cell $2$-representations}\label{scombin}

\subsection{Multisemigroup of a $2$-category}

We denote by $\mathcal{S}(\mathscr{C})$ the
{\em multisemigroup} of $\mathscr{C}$, as defined in  \cite[Section~3]{MazorchukMiemietz2};
see \cite{Viro,KudryavtsevaMazorchuk15} for more details on multisemigroups.
$\mathcal{S}(\mathscr{C})$ consists of isomorphism classes of indecomposable $1$-morphisms 
in $\mathscr{C}$ together with a formal zero element $0$. The multivalued operation 
$\star$ on $\mathcal{S}(\mathscr{C})$ is given, for indecomposable $1$-morphisms $F$ and $G$,  by:
{\small
\[
 [F]\star[G] = \begin{cases}
            0, & F \circ G \text{ is undefined}; \\
	    0, & F \circ G = 0;\\ 
            \setof{[H] \in \mathcal{S}(\mathscr{C})}{H \text{ is a summand of }F \circ G}, & \text{ otherwise}.
           \end{cases}
\]
}

\subsection{Cells in $2$-categories} 

For indecomposable $1$-morphisms $F$ and $G$, we write $F \leq_L G$ if there is a 
$1$-morphism $H$ in $\mathscr{C}$ such that $G$ is isomorphic to a direct summand of $H\circ F$.
Define $\leq_R$ and $\leq_J$ similarly using composition from the right and from both sides, respectively.
The relations $\leq_L, \leq_R, \leq_J$ are partial pre-orders on $\mathcal{S}(\mathscr{C})$, called
the {\em left}, {\em right} and {\em two-sided} pre-orders. Corresponding equivalence classes are
called {\em cells}. For a $1$-morphism $F$, we denote by $\mathcal{L}_F$,
$\mathcal{R}_F$ and $\mathcal{J}_F$ the left, right and two-sided cells containing $F$, respectively.

Following \cite[Section 4.8]{MazorchukMiemietz1}, a two-sided cell $\mathcal{J}$ in $\mathscr{C}$ is 
called \emph{strongly regular} if different left (resp. right) cells in $\mathcal{J}$ 
are not comparable with respect to $\leq_L$ (resp. $\leq_R$) and, moreover, the intersection of a 
left and a right cell in $\mathcal{J}$ is a singleton. For example, both two-sided cells 
of $\mathscr{C}_A$ in Example~\ref{CA} are strongly regular, see \cite[Section~5.1]{MazorchukMiemietz5}.

\subsection{$\mathcal{J}$-simple $2$-categories}

Given a two-sided cell $\mathcal{J}$ in $\mathcal{S}(\mathscr{C})$, we say that $\mathscr{C}$ 
is \emph{$\mathcal{J}$-simple} provided that any non-trivial $2$-ideal of $\mathscr{C}$ contains 
$\text{id}_F$ for some (and hence for all)  $F \in \mathcal{J}$. For any $2$-ideal $\mathscr{I}$
of $\mathscr{C}$, the quotient map $\mathscr{C}\to \mathscr{C}/\mathscr{I}$ induces a
partially defined bijection from the set of two-sided cells in $\mathscr{C}$ to the set of 
two-sided cells in $\mathscr{C}/\mathscr{I}$. The domain of this partial bijection consists of 
all two-sided cells whose $1$-morphisms are not sent to zero. The following
is proved in \cite[Theorem 15]{MazorchukMiemietz2}.

\begin{theorem}\label{J-simple}
Let $\mathscr{C}$ be a fiat category and $\mathcal{J}$ a non-zero two-sided cell in $\mathscr{C}$. 
Then there exists a unique $2$-ideal $\mathscr{I}$ in $\mathscr{C}$ such that $\mathscr{C}/\mathscr{I}$ 
is $\mathcal{J}$-simple.
\end{theorem}

\subsection{Cell $2$-representations}

Let $\mathcal{L}\neq \{0\}$ be a left cell in $\mathcal{S}(\mathscr{C})$ and 
$\mathtt{i}_{\mathcal{L}} \in \mathscr{C}$ be the unique object such that all 
$F \in \mathcal{L}$ have $\mathtt{i}_{\mathcal{L}}$ as the domain. Then the $2$-rep\-re\-sen\-ta\-tion
$\mathbf{N}:= \mathbf{G}_{\mathbf{P}_{\mathtt{i}_{\mathcal{L}}}}(\{\mathrm{F}:\mathrm{F}\geq_L \mathcal{L}\})$
has a unique maximal $\mathscr{C}$-stable ideal $\mathscr{J}$. The quotient
$\mathbf{G}_{\mathbf{P}_{\mathtt{i}_{\mathcal{L}}}}(\{\mathrm{F}:\mathrm{F}\geq_L \mathcal{L}\})/\mathscr{J}$
is called the {\em cell $2$-representation} corresponding to $\mathcal{L}$ and denoted
$\mathbf{C}_{\mathcal{L}}$, see \cite[Subsection~6.5]{MazorchukMiemietz2}

If $\mathscr{C}$ is weakly fiat, there is an alternative construction of $\mathbf{C}_{\mathcal{L}}$
using  $\overline{\mathbf{P}}_{\mathtt{i}_{\mathcal{L}}}$, see 
\cite[Subsection~4.5]{MazorchukMiemietz1} and  also \cite{Zhang} for some related generalizations.

\subsection{Transitive and simple transitive $2$-representations}
A finitary $2$-rep\-re\-sen\-ta\-ti\-on $\mathbf{M}$ of $\mathscr{C}$ is called \emph{transitive} 
if, for every $\mathtt{i} \in \mathscr{C}$ and every non-zero object $X \in \mathbf{M}(\mathtt{i})$, 
we have  $\mathbf{G}_{\mathbf{M}}(\{X\}) = \mathbf{M}$. All cell $2$-representations are transitive,
see \cite[Subsection~3.3]{MazorchukMiemietz5}. The following statement is proved in
\cite[Lemma~4]{MazorchukMiemietz5}:

\begin{lemma}\label{SimpTransIdeal}
Let $\mathbf{M}$ be a transitive $2$-representation of $\mathscr{C}$. 
Then the set of all $\mathscr{C}$-stable ideals of $\mathbf{M}$ which are different
from $\mathbf{M}$ contains a unique maximal element (with respect to inclusion),
denoted $\mathbf{I}$.
\end{lemma}

A transitive $2$-representation $\mathbf{M}$ is called \emph{simple transitive} if $\mathbf{I}=0$. 
The quotient $\mathbf{M}/\mathbf{I}$ is simple transitive and is 
called the \emph{simple transitive quotient of $\mathbf{M}$}.  

Cell $2$-representations are, by construction, simple transitive. However, in 
\cite[Example~3.2]{MazorchukMiemietz5} one can find examples of simple transitive
$2$-representations which are not equivalent to cell $2$-representations.

\subsection{Simple transitive subquotients of finitary $2$-representations}

For a finitary $\mathbf{M}$, denote by  $\text{Ind}(\mathbf{M})$ the set of isomorphism classes 
of indecomposable objects in $\amalg_{\mathtt{i} \in \mathscr{C}}\mathbf{M}(\mathtt{i})$. 
We say that $\mathbf{X} \subset \text{Ind}(\mathbf{M})$ is  {\em $\mathscr{C}$-stable} 
if any indecomposable direct summand of $F\;X$ is isomorphic to an object in $\mathbf{X}$, 
for any $X \in \mathbf{X}$ and any $1$-morphism $F$. For a $\mathscr{C}$-stable $\mathbf{X}$, set
$\mathbf{M_X}:=\mathbf{G}_{\mathbf{M}}(\mathbf{X})$.

Let $\mathbf{X} \subsetneq \mathbf{Y} \subset \text{Ind}(\mathbf{M})$ be $\mathscr{C}$-stable
and $\widehat{\mathbf{M_X}}$ be the $\mathscr{C}$-invariant ideal of $\mathbf{M_Y}$ generated by ${\mathbf{M_X}}$.
If $\mathbf{M_Y}/\widehat{\mathbf{M_X}}$ is transitive, then its simple transitive 
quotient is a \emph{simple transitive subquotient} of $\mathbf{M}$. By 
\cite[Theorem 8]{MazorchukMiemietz5}, the multiset of equivalence classes of
simple transitive subquotients of $\mathbf{M}$ is an  invariant of $\mathbf{M}$.

\subsection{Isotypic $2$-representations}

For a  be a finitary $2$-representation $\mathbf{M}$ of $\mathscr{C}$ and a finitary 
$\Bbbk$-linear category $\mathcal{A}$, we denote by $\mathbf{M}^{\boxtimes \mathcal{A}}$
the corresponding {\em inflation} as defined in \cite[Section 3.6]{MazorchukMiemietz6}.  
We call $\mathbf{M}$ \emph{isotypic} if all simple  transitive subquotients of $\mathbf{M}$ 
are equivalent. The following is proved in \cite[Theorem~4]{MazorchukMiemietz6}.

\begin{theorem}\label{IsoFaithClassification}
Let $\mathscr{C}$ be weakly fiat with a unique maximal two-sided cell 
$\j$, and $\l$ a left cell in $\j$.  Assume that $\j$ is strongly regular and that 
$\mathscr{C}$ is $\j$-simple. Then any isotypic faithful $2$-representation of 
$\mathscr{C}$ is equivalent to an inflation of $\mathbf{C}_{\l}$.
\end{theorem}

\section{Soergel Bimodules}\label{sSoergel}

\subsection{Soergel bimodules over the coinvariant algebra of a finite Coxeter group}
Here we describe the $2$-category of Soergel bimodules, following \cite[Example 3]{MazorchukMiemietz2}. 
For more details, see \cite{Soergel07} and \cite{Elias}. From now on we will fix 
$\mathbb{C}$ as our ground field.

Let $(W,S)$ be a finite Coxeter group with a fixed geometric representation $\mathfrak{h}$
and let $C_W$ be the corresponding coinvariant algebra. For a simple reflection $s$,
we denote by $C_W^s$ the subalgebra of $s$-invariants in $C_W$. For $w \in W$, 
fix a reduced expression $w = s_1\cdot s_2 \cdot \cdots \cdot s_k$ and define the $C_W$-$C_W$-bimodule
\[
 \hat{B}_w := C_W \otimes_{C_W^{s_1}} C_W \otimes_{C_W^{s_2}} \cdots \otimes_{C_W^{s_k}} C_W. 
\]
Set $\theta_e = C_W$. For $e \neq w \in W$, define $\theta_w$ as the unique indecomposable 
direct summand of $\hat{B}_w$ that is not isomorphic to $\theta_{w'}$ for any shorter $w'\in W$.  
It exists by \cite[Satz 6.14]{Soergel07}. The $C_W$-$C_W$-bimodule $\theta_w$ 
(which does not depend on the choice of a reduced expression above even though the $\hat{B}_w$ do) 
is the \emph{(indecomposable) Soergel bimodule} associated to $w$. 

Following \cite[Section~2.2]{MazorchukMiemietz2}, for a small category $\mathcal{A}$ 
equivalent to $C_W$-mod, let $\mathscr{S}_W$ be the fiat $2$-category defined as follows:
\begin{itemize}
\item the only object of $\mathscr{S}_W$ is $\clubsuit$, which is identified with $\mathcal{A}$,
\item $1$-morphisms in $\mathscr{S}_W$ are endofunctors of $\mathcal{A}$
given by tensoring with direct sums of  Soergel bimodules,
\item $2$-morphisms are natural transformations of functors.
\end{itemize}

\subsection{Soergel bimodules over the dihedral group $D_n$}

For $n \geq 3$, consider the $2$-category $\mathscr{S}_n$ of Soergel bimodules for
the dihedral group 
\[
  D_n = \langle s,t \; | \; s^2 = t^2 = (st)^n = (ts)^n = e \rangle=
  \{e, s, t, st, ts, sts, tst, \dots, w_0\},
\]
where $w_0$ is the (unique) longest element given by 
\begin{displaymath}
w_0 = (st)^{\frac{n}{2}} = (ts)^{\frac{n}{2}}, \text{ if $n$ is even},\qquad
w_0 = (st)^{\frac{n-1}{2}}s = (ts)^{\frac{n-1}{2}}t, \text{ if $n$ is odd.}
\end{displaymath}
As usual, we denote by $l(w)$ the length of a reduced expression of $w \in D_n$.

\subsection{Kazhdan-Lusztig basis}

Consider the group algebra $\mathbb{Z}[D_n]$. For $w\in W$, set 
\begin{displaymath}
\underline{w}:=w+\sum_{w':l(w')<l(w)}w'.
\end{displaymath}
Then $\{\underline{w}:w\in D_n\}$ coincides with the {\em Kazhdan-Lusztig basis} of $\mathbb{Z}[D_n]$,
defined in \cite{KahzdanLusztig79}, see \cite{Sauerwein15}.

\begin{example}
In the case of $D_4$, we get the following KL-basis: 
\begin{align*}
\underline{e}  & = e, \quad \underline{s} = e + s, \quad \underline{t} = e + t, \quad 
\underline{st} = e + s + t + st, \quad \underline{ts} = e + s + t + ts, \\
\underline{sts} & = e + s + t + st + ts + sts, \quad \underline{tst} = e + s + t + st + ts + tst, \\
\underline{stst} & = e + s + t + st + ts + sts + tst + stst.
\end{align*} 
\end{example}

The connection to Soergel bimodules is given by the following theorem, see \cite{Soergel92}:

\begin{theorem}[Soergel]\label{Soergel}
There is an algebra isomorphism as follows:
\begin{align*}
  \Phi: [\mathscr{S}_n(\clubsuit,\clubsuit)] &\to \Z[D_n], \\
  [\theta_w] & \mapsto \underline{w}.
\end{align*}
\end{theorem}

\subsection{The cell structure}

From \cite[Lemma~7.2]{Lusztig14}, we have: 

\begin{proposition}\label{leftMult}
For $w \in D_n$, we have
\[
 \underline{t}\cdot  \underline{w} = 
 \begin{cases}
 \underline{t}, & w=e;\\
 \underline{ts}, & w=s;\\
 \underline{s w} + \underline{tw}, & \text{ if } tw > w\text{ and } w\neq e,s;\\
 2 \underline{w}, & \text{ else},
 \end{cases}
\]
and
\[
  \underline{s}\cdot \underline{w} = 
 \begin{cases}
 \underline{s}, & w=e;\\
 \underline{st}, & w=t;\\
  \underline{sw} + \underline{tw}, & \text{ if } sw > w\text{ and } w\neq e,t;\\
  2 \underline{w}, & \text{ else}.
 \end{cases}
\]
\end{proposition}

Applying $w\mapsto w^{-1}$ to Proposition~\ref{leftMult}, we get a similar result for right 
multiplication with $\underline{s}$ and $\underline{t}$, respectively. 

We write $\mathcal{L}_w$for $\mathcal{L}_{\theta_w}$, the left cell of $\theta_w$ etc. Now, 
the cell structure of $\mathscr{S}_n$ is given as follows, see also \cite[Section 8.7.]{Lusztig14}.
\begin{center}
\begin{picture}(200,100)(0,0)
 \put(90,80){\framebox(20,20){$\theta_e$}}
 \put(40,59.5){\framebox(60,20){$\theta_s,\theta_{sts},\dots$}}
 \put(100,59.5){\framebox(60,20){$\theta_{st},\theta_{stst},\dots$}}
 \put(40,39.5){\framebox(60,20){$\theta_{ts},\theta_{tsts},\dots$}}
 \put(100,39.5){\framebox(60,20){$\theta_t,\theta_{tst},\dots$}}
 \put(90,19){\framebox(20,20){$\theta_{w_0}$}}
 \put(60,85){$\mathcal{L}_s$}
 \put(120,85){$\mathcal{L}_t$}
 \put(25,45){$\mathcal{R}_t$}
 \put(25,65){$\mathcal{R}_s$}
\end{picture}
\end{center}

\begin{example}
 In the case of $\mathscr{S}_4$, this gives us the following cell structure
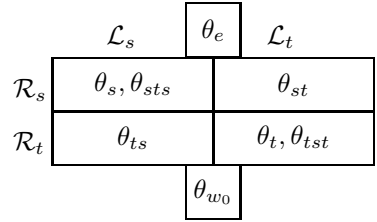
\newline
\begin{center}
\begin{picture}(200,100)(0,0)
 \put(90,80){\framebox(20,20){$\theta_e$}}
 \put(40,59.5){\framebox(60,20){$\theta_s,\theta_{sts}$}}
 \put(100,59.5){\framebox(60,20){$\theta_{st}$}}
 \put(40,39.5){\framebox(60,20){$\theta_{ts}$}}
 \put(100,39.5){\framebox(60,20){$\theta_t,\theta_{tst}$}}
 \put(90,19){\framebox(20,20){$\theta_{w_0}$}}
 \put(60,85){$\mathcal{L}_s$}
 \put(120,85){$\mathcal{L}_t$}
 \put(25,45){$\mathcal{R}_t$}
 \put(25,65){$\mathcal{R}_s$}
\end{picture}
\end{center}
and the two-sided cells are
\begin{displaymath}
\mathcal{J}_1 = \{\theta_e\},\quad \mathcal{J}_2 = \{\theta_s, \theta_t, \theta_{st}, 
\theta_{ts}, \theta_{sts}, \theta_{tst}\},\quad \mathcal{J}_3 = \{\theta_{w_0} = \theta_{stst} = \theta_{tsts}\}. 
\end{displaymath}
We see that $\mathcal{J}_1$ and $\mathcal{J}_3$ are strongly regular whereas $\mathcal{J}_2$ 
is not as, for example,   the intersection of $\mathcal{L}_s$ and $\mathcal{R}_s$ is not a singleton.
\end{example}

\subsection{$D_n$-modules} \label{dnModules}
The representation theory of $D_n$ is well-known, see e.g. \cite{Weintraub03}. 

Recall that, if $n$ is even, there are four non-isomorphic $1$-dimensional modules, denoted 
$V_{\varepsilon,\delta}$, for $\varepsilon,\delta \in \{-1,1\}$, where $s$ acts 
via $\varepsilon$ and $t$ via $\delta$. If $n$ is odd, there are two 
simple $1$-dimensional modules, namely, $V_{1,1}$ and $V_{-1,-1}$. 

Simple $2$-dimensional modules are $V^{(n,k)}$, where $k\in\{1,2,\dots,\frac{n-2}{2}\}$
for even $n$, and in $\{1,2,\dots,\frac{n-1}{2}\}$ for odd $n$. The matrices of 
$\underline{s}$ and $\underline{t}$ acting on $V^{(n,k)}$ are:
\[
\begin{pmatrix}2 & 0\\0 & 0\end{pmatrix}\qquad\text{ and }\quad
\begin{pmatrix}1+\cos(\frac{2k\pi}{n}) & \sin(\frac{2k\pi}{n}) \\
\sin(\frac{2k\pi}{n}) & 1-\cos(\frac{2k\pi}{n})\end{pmatrix},\qquad\text{ respectively}.
\]
The characteristic polynomial for the element $\underline{s} + \underline{t}$ is thus
\begin{equation}\label{charPolyGroup}
\chi_{n,k}(x) = x^2 -4x + 2 - 2\cos\left(\frac{2k\pi}{n}\right).
\end{equation}

\subsection{The cell $2$-representations for the one element cells}\label{trivialCellModules}

We start by describing cell $2$-representations for the singleton left cells  $\mathcal{L}_e$
and $\mathcal{L}_{w_0}$.
 
\begin{proposition}\label{1ElementCells-1}
For $n\geq 3$, we have $[\mathbf{C}_{\l_e}]^{\C} \simeq V_{-1,-1}$. 
\end{proposition}

\begin{proof}
For $\mathcal{L}=\mathcal{L}_e$, each $1$-morphism $F$ satisfies $F \geq_L \l_e$. Thus 
$\mathbf{N}(\c) = \mathscr{S}_n(\c, \c)$. Let $\mathbf{I}$ be the ideal given by Lemma~\ref{SimpTransIdeal}. 
By \cite[Lemma~16(i)]{MazorchukMiemietz2}, if $\text{id}_F$ belongs to $\mathbf{I}$ for some $F$, 
then $\text{id}_G \in \mathbf{I}$ for all $G \geq_L F$. Therefore $\text{id}_F\in\mathbf{I}$ 
for all $F >_L \theta_e$. This means  that both $\theta_s$ and 
$\theta_t$ annihilate $\mathbf{C}_{\mathcal{L}}(\clubsuit)$. Hence both $\underline{s}$ and 
$\underline{t}$ annihilate $[\mathbf{C}_{\l_e}]^{\C}$ which implies 
$[\mathbf{C}_{\mathcal{L}_e}]^{\C} \simeq V_{-1,-1}$.
\end{proof}

\begin{proposition}\label{1ElementCells-2}
For $n\geq 3$, we have that $\mathbf{C}_{\l_{w_0}}$ is equivalent to the natural action of 
$\mathscr{S}_n$ on the additive category of projective objects in $\mathcal{A}$.
Moreover,  $[\mathbf{C}_{\l_{w_0}}]^{\C} \simeq V_{1,1}$. 
\end{proposition}

\begin{proof}
We start by collecting some facts about $C_W$ that will help us to describe $\mathbf{C}_{\l_{w_0}}$. 
From \cite[Claim~2.2.1]{Elias} it follows that $C_W$ is free of rank $2$ over both $C_W^s$  and $C_W^t$.
Moreover, $C_W$ is local since it is a non-negatively graded algebra whose degree zero part is simple and which
has finite dimension $2l(w_0) = 2n$, see, for example, \cite[Corollary~3.10]{Hiller82}. Thus there is a 
unique, up to isomorphism,  indecomposable projective module $P$ in $C_W$-mod. Let $L$ be its simple top. 
 
Next we  claim that $\theta_{w_0}\cong C_W \otimes C_W$. This must be well-known for any $W$; 
the case of the dihedral group is stated in \cite{EliasWilliamson14}. However, we did not manage to 
find a full proof of this fact and hence we give a proof in the dihedral case here. 
We consider the defining $2$-representation of 
$\mathscr{S}_n$, that is the natural action of $\mathscr{S}_n$ on $\mathcal{A}$. Via the equivalence 
between $\mathcal{A}$ and $C_W$-mod, this gives a weak action of $\mathscr{S}_n$ on $C_W$-mod.
 
\begin{lemma}\label{simpeltTop}
Let $L$ denote the unique, up to isomorphism, simple $C_W$-module.  
\begin{enumerate}[$($a$)$]
\item\label{simpeltTop.1} We have $\text{dim}(\theta_w\: L) = 2l(w)$, for all $w \neq e$.
\item\label{simpeltTop.2} The module $\theta_w\: L$ has simple top, for all $w \in D_n$. In particular,
the module $\theta_w\: L$ is indecomposable.
\end{enumerate}
\end{lemma}
 
\begin{proof}
  To prove claim~\eqref{simpeltTop.1}, we proceed by induction on the length of $w$. 
  To prove the basis of the induction, let  $w = s$. Then 
 \[
  \theta_s\: L = C_W \otimes_{C_W^s} C_W \otimes_{C_W} L \simeq C_W \otimes_{C_W^s} L.
 \]
 Since $C_W$ is free of rank $2$ over $C_W^s$, we get that $\text{dim}(\theta_sL) = 2 = 2l(s)$,
 as required. Similar arguments work in the case $w = t$. 
 
 Now assume claim~\eqref{simpeltTop.1} holds for $w$ such that $l(w) \leq k$. 
 Let $w'$ be an element of length $k+1$. Then we either have $w'=sw$ or we have $w'=tw$
 for some element $w$ of length $k$. We consider the case $w'=sw$, the other case being similar.
 We have
 \begin{equation}\label{eqnn2}
  \theta_s\theta_w \:L = \theta_{sw} \:L \oplus \theta_{tw}\: L
 \end{equation}
 by Proposition~\ref{leftMult} and Theorem~\ref{Soergel}.
 Since $\theta_s$ is exact (as $\mathscr{S}_n$ is fiat), using the basis of the induction and the inductive assumption, we have 
 \begin{equation}\label{eqnn3}
 \text{dim}(\theta_s\theta_w L) = 2 \text{dim}(\theta_w L) = 4l(w). 
 \end{equation}
 Combining \eqref{eqnn2} with \eqref{eqnn3} and using the inductive assumption, we have
 \[
  \text{dim}(\theta_{sw} L) = 4l(w) - \text{dim}(\theta_{tw}L) = 4l(w) - (2(l(w) - 1)) = 2(l(w)+1) = 2l(sw).
 \]
 This proves claim~\eqref{simpeltTop.1}.
 
 Claim~\eqref{simpeltTop.2} is obvious for $w=e$ since $\theta_e L = L$. To prove claim~\eqref{simpeltTop.2}
 for other $w$, we again proceed by induction on $l(w)$. If $w=s$, we note that $\theta_s\: L$ is isomorphic,
 as a $C_W$-module, to $\mathrm{Ind}^{C_W}_{C_W^s}\: L^{(s)}$, where $L^{(s)}$ is the unique, up to isomorphism,
 simple $C_W^s$-module (note that $\dim(L^{(s)})=1$). Therefore, by adjunction,
 \begin{displaymath}
 \mathrm{Hom}_{C_W}(\mathrm{Ind}^{C_W}_{C_W^s}\: L^{(s)},L)=
 \mathrm{Hom}_{C_W^s}(L^{(s)},\mathrm{Res}^{C_W}_{C_W^s}\: L).
 \end{displaymath}
 The space on the right hand side is $1$-dimensional as $L^{(s)}$ is $1$-dimensional. 
 This implies claim~\eqref{simpeltTop.2}
 for $w=s$. For $w=t$, we can use a similar argument.

 Before we can make a general induction step, we have to consider the cases $w=st$ and $w=ts$. We consider
 the first case, the second one is similar. By adjunction,  
 \begin{equation}\label{eqnn4}
  \text{Hom}_{C_W}(\theta_s\theta_t\: L,L) = \text{Hom}_{C_W}(\theta_t \:L,\theta_s\: L).
 \end{equation}
 By the previous paragraph, both $\theta_t \:L$ and $\theta_s\: L$ are  indecomposable modules
 of dimension two with simple top and socle isomorphic to $L$. Therefore the right hand side of \eqref{eqnn4} 
 is one-dimensional (which is exactly what we need) unless $\theta_t \:L\simeq \theta_s\: L$. To prove that 
 $\theta_t \:L\not\simeq \theta_s\: L$, it is enough to show that these two modules are annihilated by different 
 elements of $\mathfrak{h}$.  That, in turn, is equivalent to the fact that $\mathfrak{h}$ does not have non-zero
 elements whose linear span is invariant with respect to both $s$ and $t$. The latter means exactly that the 
 $W$-module $\mathfrak{h}$ is simple, which is the case.

 To make the general induction step, we can now assume that claim~\eqref{simpeltTop.2} is true for all $w$
 such that $l(w)\leq k$ and that $k\geq 2$. Let $w'$ be of length $k+1$. Then we can write $w'=sw$
 or $w'=tw$ for some $w$ of length $k$. We consider the case $w'=sw$, the other one being similar.
 We have
 \[
  \text{dim}(\text{Hom}_{C_W}(\theta_s\theta_w\: L,L)) = \text{dim}(\text{Hom}_{C_W}(\theta_w \:L,\theta_s\: L))
 \]
 by adjunction. At the same time, we also have $2\leq \text{dim}(\text{Hom}_{C_W}(\theta_s\theta_w\: L,L))$ as
 $\theta_s\theta_w \:L = \theta_{sw}\:L \oplus \theta_{tw}\:L$ by Proposition~\ref{leftMult}.
 Moreover, $\text{dim}(\text{Hom}_{C_W}(\theta_w \:L,\theta_s\: L))\leq 2$ as $\theta_w \:L$ has simple
 top by inductive assumption and $\theta_s\: L$ has dimension two by claim~\eqref{simpeltTop.1}.
 Therefore $\text{dim}(\text{Hom}_{C_W}(\theta_s\theta_w\: L,L))=2$. As 
 $\theta_s\theta_w \:L = \theta_{sw}\:L \oplus \theta_{tw}\:L$ and both summands are non-zero,
 each of the summands must contribute with at least one homomorphism. Therefore
 $\text{dim}(\text{Hom}_{C_W}(\theta_{sw}\: L,L)) = 1$, as required.
 \end{proof}
 
 \begin{corollary}\label{longestbimodule}
 We have $\theta_{w_0} \simeq C_W\otimes C_W$.
 \end{corollary}

 \begin{proof}
 By Lemma~\ref{simpeltTop}\eqref{simpeltTop.1}, we have $\text{dim}(\theta_{w_0}\:L) = 2l(w_0) = 2n = \text{dim}(P)$.
 Furthermore, by Lemma~\ref{simpeltTop}\eqref{simpeltTop.2}, the module $P$ surjects onto $\theta_{w_0}\:L$. 
 Therefore we have $P\simeq \theta_{w_0}\: L$. Since $\theta_{w_0}$ is indecomposable and $\mathscr{S}_n$ 
 is fiat, we have $\theta_{w_0}$ is isomorphic to tensoring with a projective $C_W$-$C_W$-bimodule by 
 \cite[Lemma~13]{MazorchukMiemietz5}. Thus $\theta_{w_0}$ is isomorphic to $C_W \otimes C_W$, as claimed.
 \end{proof}
 
 Now, to describe $\mathbf{C}_{\l_{w_0}}$, we need to describe the corresponding  $\mathbf{N}$ and $\mathbf{I}$. As
 the left cell $\{\theta_{w_0}\}$ is maximal with respect to the left order, we have 
 $\mathbf{N} = \text{add}(\{\theta_{w_0}\})$. 
 Note that, if $I$ is an ideal in $C_W$, then $C_W \otimes I$ is a left $2$-ideal of $C_W \otimes C_W$.
 This implies $\mathbf{I} \supset C_W \otimes \text{Rad}(C_W)$. 
 
 Consider the $2$-natural transformation $\Phi$ from  $\mathbf{P}_{\c}$ to $\mathcal{A}$ which sends $\mathbbm{1}_{\c}$ to $L$.
 Then $\Phi$, clearly, annihilates $C_W \otimes \text{Rad}(C_W)$ (viewed as endomorphism of 
 $\mathbbm{1}_{\c}$). Moreover, $\Phi$ maps $\mathbf{N}$ to the category $\mathcal{A}_{\mathrm{proj}}$ 
 of projective objects in $\mathcal{A}$. 
 Thus $\Phi$ induces a $2$-natural transformation from $\mathbf{C}_{\l_{w_0}}$ to 
 $\mathcal{A}_{\mathrm{proj}}$ which maps $\theta_{w_0}$ to $\theta_{w_0}\: L\simeq P$. By 
 construction, this $2$-natural transformation is an equivalence.
 
 As $\theta_s\theta_{w_0} = \theta_t\theta_{w_0} = \theta_{w_0} \oplus \theta_{w_0}$, we see that 
 $\underline{s}$ and $\underline{t}$ act as the scalar $2$ on $[\mathbf{C}_{\mathcal{L}_{w_0}}]^{\C}$. 
 Since $e$ acts as the identity, we see that both $s$ and $t$ act as the identity and hence 
 $[\mathbf{C}_{\mathcal{L}_{w_0}}]^{\C} \simeq V_{1,1}$. This completes the proof.
 \end{proof}

\subsection{Decategorification of the cell $2$-representations of $\mathscr{S}_4$}
\label{DecatS4}

Let us now describe the decategorifications of the cell $2$-representations of 
$\mathscr{S}_4$. As we have seen above, $\mathscr{S}_4$ has the following $4$ left cells.
\[
  \mathcal{L}_e = \{\theta_e\},\, \mathcal{L}_s = \{\theta_s, \theta_{ts}, \theta_{sts}\}, \,\mathcal{L}_t = \{\theta_t, \theta_{st},\theta_{tst}\},\, \mathcal{L}_{w_0} = \{\theta_{w_0} = \theta_{stst} = \theta_{tsts}\}
\]
Moreover, from the previous subsection we have that 
\[
 [\mathbf{C}_{\mathcal{L}_e}]^{\C} \simeq V_{-1,-1}, \quad [\mathbf{C}_{\mathcal{L}_{w_0}}]^{\C} \simeq V_{1,1}.
\]

Consider $\l_s$ and denote by $M = [\mathbf{C}_{\mathcal{L}_s}]^{\C}$ the decategorification of $\mathbf{C}_{\mathcal{L}_s}$. Then $M$ is generated by 
$[\theta_s], [\theta_{sts}], [\theta_{ts}]$, and in this basis
we can use Proposition~\ref{leftMult} to get the following matrices for the action of $[\theta_s]$ and $[\theta_t]$, respectively:
\begin{displaymath}
\begin{pmatrix}
                2 & 0 & 1 \\
		0 & 2 & 1 \\
		0 & 0 & 0 
                \end{pmatrix} \qquad
\begin{pmatrix}
                0 & 0 & 0 \\
		0 & 0 & 0 \\
		1 & 1 & 2 
                \end{pmatrix}.
\end{displaymath}
The characteristic polynomials are $p_s(x) = x(x-2)^2$  for $[\theta_s]$ and $p_t(x) = x^2(x-2)$  for $[\theta_t]$. 
This yields that $M$ can only be isomorphic to either $V_{1,1} \oplus V_{1,-1} \oplus V_{-1,-1}$ or 
$V^{(4,1)}_1 \oplus V_{1,-1}$. The first case can be excluded
since the principal $2$-representation decategorifies to the regular representation of $D_4$ which implies
that the simple $2$-dimensional $D_4$-module appears with multiplicity two in the union of all cell $2$-representations.
Thus $M \simeq V^{(4,1)}_1 \oplus V_{1,-1}$. A similar argument shows that 
$[\mathbf{C}_{\mathcal{L}_t}]^{\C} \simeq V^{(4,1)} \oplus V_{-1,1}$.
In particular, the decategorifications of $\mathbf{C}_{\l_s}$ and $\mathbf{C}_{\l_t}$ are not isomorphic 
as $D_4$-modules. Consequently, we get the following:

\begin{proposition}\label{LeftCellsNotEquivalent}
The $2$-representations $\mathbf{C}_{\l_s}$ and $\mathbf{C}_{\l_t}$ are not equivalent.
\end{proposition}

\section{Simple transitive $2$-representations of $\mathscr{S}_n$}\label{ssimple}

\subsection{Irreducible matrices and Perron-Frobenius Theorem}

In this section we study simple transitive $2$-representations of $\mathscr{S}_n$, where $n\geq 3$. 

We define the \emph{rank} of a $2$-representation $\mathbf{M}$ as the number of 
isomorphism classes of indecomposable objects in $\mathbf{M}(\c) =: \mathcal{B}$. 

Fix a collection $X_1, \dots, X_n$ of representatives of isomorphism classes  of indecomposable objects
in $\mathcal{B}$. Then the combinatorics of the action of any $1$-morphism $F$ on $\mathcal{B} = \mathbf{M}(\c)$ 
is encoded in the matrix $\Lparen F\Rparen$. Observe that $\Lparen F\Rparen$ has non-negative integer
entries. Hence, we can use the results by Perron and Frobenius on the structure of 
non-negative matrices, see \cite{Frobenius08, Frobenius09, Perron07} and also \cite{Meyer00} for 
a modern version. 

A matrix $Q = (a_{ij}) \in \R^{n\times n}$ is called \emph{non-negative} if all 
$a_{ij}\geq 0$. For a non-negative $A$, its \emph{action graph} $G_Q$
has vertices $\{1, 2, \dots, n\}$ and a directed edge from $i$ to $j$ if $a_{ji} \neq 0$. 
$Q$ is called  \emph{irreducible} if its action graph is strongly connected. 

\begin{theorem}[Perron-Frobenius]\label{PerronFrobenius}
Each non-negative and irreducible $Q \in \R^{n\times n}$ has 
a positive real eigenvalue $\lambda$ such that any other (complex) 
eigenvalue  $\mu$ of $Q$ satisfies $|\mu|<|\lambda|$. Moreover,
the algebraic multiplicity of $\lambda$ is one.
\end{theorem}

\begin{proposition}\label{transitiveIsStronglyConnected}
For any transitive $2$-representation $\mathbf{M}$ of $\mathscr{S}_n$, we have that
the matrix $\Lparen \mathbf{M}(\theta_s \oplus \theta_t)\Rparen$ is non-negative
and irreducible. 
\end{proposition}

\begin{proof}
The non-negativity of $Q:=\Lparen \mathbf{M}(\theta_s \oplus \theta_t)\Rparen$ is already explained above.
Let $1 \leq i,j \leq k$ with $i\neq j$. As $\mathbf{M}$ is transitive, there is $\theta_w$,
with $w\neq e$, such that $X_j$ is isomorphic to a direct summand of $\theta_w X_i$.
If $w=s_1s_2\cdots s_m$ is a reduced expression, then $\theta_w$ is a direct summand of 
$\theta_{s_1}\theta_{s_2}\cdots \theta_{s_m}$ and hence also of $(\theta_s \oplus \theta_t)^m$.
This implies that $G_Q$ is strongly connected and hence  $Q$ is irreducible. 
\end{proof}

\subsection{Rough combinatorics of simple transitive 
$2$-representations of $\mathscr{S}_n$}\label{sStructofMat}

Let $\mathbf{M}$ be a simple transitive $2$-rep\-re\-sen\-ta\-ti\-on of $\mathscr{S}_n$. 
If we assume that $\theta_{w_0}\mathbf{M}(\c) \neq 0$, then, by \cite[Theorem 4]{MazorchukMiemietz6}, 
we get  $\mathbf{M} \simeq \mathbf{C}_{\mathcal{L}_{w_0}}$. If, on the other hand, 
we have  $(\theta_s \oplus \theta_t)\:\mathbf{M}(\c) = 0$, 
then $\mathbf{M} \simeq \mathbf{C}_{\mathcal{L}_e}$ 
by the discussion in Subsection~\ref{trivialCellModules}.

So, from now on, we assume $\theta_{w_0}\:\mathbf{M}(\c) = 0$ and 
$(\theta_s \oplus \theta_t)\:\mathbf{M}(\c) \neq 0$. 
Let $X_1, \dots, X_r$ be a complete and irredundant list of pairwise non-isomorphic 
indecomposable objects in $\mathbf{M}(\clubsuit)$. We will in some of the proofs consider 
the abelianization $\overline{\mathbf{M}}$ of $\mathbf{M}$ and then denote the   
indecomposable projective object $0 \to X_i$ in $\overline{\mathbf{M}}(\c)$ by $P_i$. 
Moreover, for $i=1,2,\dots,r$, we denote by $L_i\in \overline{\mathbf{M}}(\c)$ the simple top of $P_i$.

\begin{lemma}\label{L3}
There exists an ordering of $X_1, \dots, X_r$ such that 
\begin{equation*}
\Lparen\theta_s\Rparen = \left(
\begin{array}{c|c}
2I_k & B \\\hline
0 &   0  
\end{array} 
\right),
\end{equation*}
where $1\leq k \leq r$, the matrix $B$ is non-negative and $I_k$ is the 
$(k\times k)$-identity matrix.
\end{lemma}

\begin{proof} 
For $i=1,\dots,r$, write $\theta_s X_i=a_i X_i\oplus Y_i$, where $a_i\in\mathbb{Z}_{\geq 0}$  and 
$Y_i$ does not have $X_i$ as a summand. We want to prove that either $a_i = 2$ and $Y_i = 0$ 
or $a_i=0$ and every summand $X_j$ of $Y_i$ satisfies $\theta_s X_j = 2 X_j$. 
From $\theta_s^2=\theta_s\oplus\theta_s$, we have
\begin{equation}\label{eq2508-1}
a_i^2 X_i \oplus a_i Y_i  \oplus \theta_s\,Y_i\simeq 2a_i X_i\oplus 2 Y_i.
\end{equation}
If $a_i>2$, we have that the multiplicity $2a_i$ of $X_i$ on the right hand side is strictly smaller
than the multiplicity, which is at least $a_i^2$, of $X_i$ on the left hand side, a contradiction. 
Therefore $a_i\in\{0,1,2\}$. If $a_i=2$, then \eqref{eq2508-1} implies $\theta_s\,Y_i=0$.
If $a_i=0$, then \eqref{eq2508-1} implies $\theta_s\,Y_i=2 Y_i$.

Consider the case $a_i=1$. Then \eqref{eq2508-1} implies $\theta_s\,Y_i=X_i\oplus  Y_i$.
This means that there is a unique indecomposable direct summand $X_j$ of $Y_i$ such that 
$\theta_s\,X_j=X_i\oplus  Z$. Note that $i\neq j$.
Write $Y_i= X_j\oplus U$. Then neither $X_i$ nor $X_j$ are direct summands
of $U$. We claim that neither $X_i$ or $X_j$ are summands of $\theta_s\,U$. For this
it is enough to show that $X_j$ is not a summand of $\theta_s\,U$. If 
$X_j$ is a direct summand of $\theta_s\,U$, then $X_j$ is not a direct summand of $Z$. 
Therefore $X_j$ is not a direct summand of
\begin{displaymath}
\theta^2_s\,X_j=\theta_s\,X_j\oplus \theta_s\,X_j\cong X_i\oplus  Z \oplus X_i\oplus  Z 
\end{displaymath}
either. This, however, contradicts the fact that $X_j$ is a direct summand of $\theta_s\,X_i$
and the latter is a direct summand of $\theta^2_s\,X_j=\theta_s\,(X_i\oplus  Z)$.
Hence $\theta_s\,U$ is in the additive closure of $U$, moreover, 
$\theta_s\, X_i=X_i\oplus X_j\oplus U$ and $\theta_s\, X_j=X_i\oplus X_j\oplus U'$,
where $U'$ is in the additive closure of $U$.

Consider the fiat $2$-full $2$-subcategory $\mathscr{C}$ of $\mathscr{S}_n$ whose indecomposable 
$1$-mor\-phisms are all $1$-morphisms isomorphic to $\theta_e$ and $\theta_s$. Note that, clearly, 
all two-sided cells in $\mathscr{C}$ are strongly regular. Consider the 
$2$-representation $\mathbf{N}$ of $\mathscr{C}$ given by restricting the action of $\mathscr{C}$ to the additive
closure of $X_i\oplus X_j\oplus U$ (the latter is closed under the action of $\mathscr{C}$  by the computation in the
previous paragraph). Let $\mathbf{I}$ be the ideal in $\mathbf{N}$ generated by
$\mathrm{id}_U$. Then the representation $\mathbf{N}/\mathbf{I}$ is transitive and hence has a simple
transitive top. From the previous paragraph, we see that the matrix of $[\theta_s]$ for this
simple transitive $2$-representation is
\begin{displaymath}
\left(\begin{array}{cc}1&1\\1&1\end{array}\right). 
\end{displaymath}
A simple transitive $2$-representation of $\mathscr{C}$  is equivalent to a cell two-representation
by \cite[Theorem~18]{MazorchukMiemietz5}. Therefore the matrix of $[\theta_s]$ for a simple transitive 
$2$-rep\-re\-sen\-ta\-tion can only be either the matrix $(0)$ (in case of the cell $2$-representation corresponding
to $\theta_e$) or $(2)$ (in case of the cell $2$-representation corresponding
to $\theta_s$). This is a contradiction which shows that the
case $a_i=1$ cannot occur.

If $a_i=2$, then, as mentioned above, $\theta_s\,Y_i=0$. Let $X_j$ be a non-zero direct summand of $Y_i$. Then,
by adjunction, we have 
\begin{displaymath}
0\neq \mathrm{Hom}_{\overline{\mathbf{M}}(\c)}(\theta_s\, X_i,L_j)=
\mathrm{Hom}_{\overline{\mathbf{M}}(\c)}(X_i,\theta_s\, L_j)
\end{displaymath}
and hence $\theta_s\, L_j\neq 0$. Therefore $\theta_s\, X_j\neq 0$ as $\theta_s$ is exact. This 
contradicts $\theta_s\,Y_i=0$ and implies that $Y_i=0$. 

Finally, assume that $a_i=0$. As mentioned above, in this case we have $\theta_s\,Y_i=2 Y_i$.
Let $X_j$ be a direct summand of $Y_j$. Then the adjunction argument from the previous
paragraph implies $\theta_s\, L_j\neq 0$ and hence $\theta_s\, X_j\neq 0$.
Write $Y_i=U\oplus V$, where, for each direct summand $X_j$ of $U$, we have $a_j=2$ while for
for each direct summand $X_j$ of $V$ we have $a_j=0$. Then $U\oplus U$ is a direct summand of $\theta_s\, U$
and hence $\theta_s\, V$ belongs to the additive closure of $V$. 

We claim that $V=0$. Define $I\subset \{1,2,\dots,r\}$ via 
$\mathrm{add}(V)=\mathrm{add}(\{X_i:i\in I\})$.
We want to show that $I=\varnothing$. If $I\neq \varnothing$, let $I'\subset I$ be minimal, with respect to inclusions,
such that $\mathrm{add}(\{X_i:i\in I'\})$ is $\theta_s$-invariant. Assume that $I'$ consists
of more than one element and let $m\in I'$. From the definition of $V$ and $I'$ we have that 
$\theta_s\,X_m$ belongs to the additive closure of $\{X_i:i\in I',i\neq m\}$.
By the minimality of $I'$, we have $\theta_s\,X_m\neq 0$. However, as
$\theta_s^2=\theta_s\oplus\theta_s$, the additive closure of $\theta_s\,X_m$ must be
$\theta_s$-invariant, which again contradicts minimality of $I'$. Therefore $I'=\{i\}$, for some $i$,
and hence $\theta_s\,X_i=0$ by the definition of $V$. This contradicts $\theta_s\, X_i\neq 0$
established in the previous paragraph. Therefore $V=0$.

Now we see that, if we first take all $X_j$ such that $a_j=2$ and then all $X_j$
such that $a_j=0$, then $\Lparen \theta_s\Rparen$ will have the required form. This completes the proof.
\end{proof}

By symmetry, the analogous result holds for $\theta_t$, however, the orderings on the basis
elements for $\theta_s$ and $\theta_t$ might be different. 

\begin{lemma}\label{L4}
There exists no $1 \leq i \leq r$ such that 
$\theta_s\:P_i = 2 P_i = \theta_t\:P_i$.
\end{lemma}

\begin{proof}
If such $P_i$ exists, then $\text{add}(\{P_i\})$
is closed under $\theta_s$ and $\theta_t$ and thus
under $\mathscr{S}_n$. From the transitivity of ${\mathbf{M}}$,
we get ${\mathbf{M}}(\c)=\text{add}(\{P_i\})$ and $r=1$. We also have
$[\theta_s]=[\theta_t]=(2)$, which implies that $\mathbf{M}$ decategorifies
to the trivial $D_n$-module. The latter, however, is not annihilated by 
$\underline{w_0}$.  Thus $\theta_{w_0}\: \mathbf{M}(\c) \neq  0$, a contradiction.
\end{proof}

\begin{corollary}\label{Cor5}
There exists an ordering of $X_1, \dots, X_r$ such that 
\begin{equation*}
\Lparen\theta_s \oplus \theta_t\Rparen = \left(
\begin{array}{c|c}
 2I_k & {B} \\ \hline
 B' & 2I_{r-k}
\end{array}
\right),
\end{equation*}
where $B$ and $B'$ are non-negative and $1 \leq k < r$.
\end{corollary}

\begin{proof}
By Lemma \ref{L3}, there is  an ordering of $X_1, \dots X_r$ and a matrix $B$ such that 
\begin{equation*}
 \Lparen\theta_s\Rparen = \left(
    \begin{array}{c|c}
     2I_k & B \\\hline
      0 &   0  
    \end{array} 
\right).
\end{equation*}
Moreover, Lemma \ref{L3} for $[\theta_t]$ and Lemma \ref{L4} imply that 
in the same  ordering holds
\begin{equation*}
\Lparen\theta_t\Rparen = \left( 
\begin{array}{c|c}
0 & 0 \\ \hline
B' & 2I_l
\end{array}
\right)
\end{equation*}
for some $l>0$ such that $k+l\leq r$. If $k + l < r$, then there exists a zero row in the  
matrix $[\theta_t \oplus \theta_s]$  which implies that no power of $[\theta_s \oplus \theta_t]$ 
can be  totally positive. This contradicts transitivity of $\mathbf{M}$ and thus $l = r-k$.
\end{proof}

\subsection{The $D_3$-case}
                                             
The special case $n = 3$ is the case of Soergel 
bimodules in type $A_2$. By \cite[Theorem 18]{MazorchukMiemietz5}, any 
simple transitive $2$-representation of $\mathscr{S}_3$ is equivalent to a cell 
$2$-representation. The cell $2$-representation 
$\mathbf{C}_{\mathcal{L}_s}\cong \mathbf{C}_{\mathcal{L}_t}$ decategorifies to the
unique simple $2$-dimensional $D_3$-module.

\subsection{First results for simple transitive $2$-representations}

From now on we assume $n > 3$. Our first result  is 
a generalization of \cite[Proposition~21]{MazorchukMiemietz5}.

\begin{theorem}\label{noSimple}
Let $\mathbf{M}$ be a finitary $2$-representation of $\mathscr{S}_n$. 
If the $[\mathscr{S}_n(\clubsuit,\clubsuit)]^{\C}$-module $[\mathbf{M}(\c)]^{\C}$ is simple, 
then $[\mathbf{M}(\c)]^{\C} \simeq V_{1,1}$ or $[\mathbf{M}(\c)]^{\C} \simeq V_{-1,-1}$.
\end{theorem}

\begin{proof}
Recall that we have exactly three two-sided cells in $\mathscr{S}_n$, namely
\[
\mathcal{L}_e = \mathcal{J}_1 = \{\theta_e\}, \quad \mathcal{J}_2 = 
\setof{\theta_w}{1\leq l(w) \leq n-1}, \quad \mathcal{L}_{w_0} = \mathcal{J}_3 = \{\theta_{w_0}\}.
\]
From above, we know that $\mathcal{J}_2$ consists of two left cells $\mathcal{L}_s$ and $\mathcal{L}_t$.
Furthermore, as we have seen above, the cell $2$-representations $\mathbf{C}_{\mathcal{L}_e}$ and 
$\mathbf{C}_{\mathcal{L}_{w_0}}$ categorify $V_{-1,-1}$ and $V_{1,1}$, respectively. 
Note that, by the same argument as in \cite[Proposition 22]{MazorchukMiemietz5}, we get that 
$[\mathbf{M}(\c)]^{\C} \ncong V_{1,-1}$ and $[\mathbf{M}(\c)]^{\C} \ncong V_{-1,1}$, 
which of course only could occur in case $n$ is even. 
Thus it is left to prove that $[\mathbf{M}(\c)]^{\C}$ is not isomorphic to 
any of the $2$-dimensional simple $\C[D_n]$-modules described in Subsection~\ref{dnModules}.

Assume that $[\mathbf{M}(\c)]^{\C} \simeq V^{(n,k)}$ for some $1 \leq k \leq \frac{n-1}{2}$.
Then we have that the matrix $X := \Lparen\theta_s \oplus \theta_t\Rparen$ has
the same characteristic polynomial as the element $\underline{s} + \underline{t}$ acting on $V^{(n,k)}$. 
However, this polynomial has to have integer coefficients since $X$, 
by construction, has only non-negative integer coefficients.
Hence, we have to check for which $k$ the polynomial
\[
 \chi_{n,k}(x) = x^2 - 4x  + 2 - 2\cos\left(\frac{2k\pi}{n}\right) 
\]
has integer coefficients, that is when $2\cos(\frac{2k\pi}{n}) \in \Z$. 
This is equivalent to asking for which $n$ and $1 \leq k \leq \frac{n-1}{2}$ 
we have $\cos(\frac{2k\pi}{n}) \in \frac{1}{2}\Z \cap (-1,1) = 
\{-\frac{1}{2},0,\frac{1}{2}\}$. 

In the {\bf first} case, that is  when $\cos(\frac{2k\pi}{n}) = -\frac{1}{2}$, we get 
$k = \frac{n}{3}$ and thus $n \in 3\Z$. In this case, $\chi_{\frac{n}{3}}(x) = x^2 - 4x + 3 = (x-3)(x-1)$ 
and thus $X$ has eigenvalues $1$ and $3$, trace $4$ and determinant $3$. 
This means that, up to reordering of the basis elements, $X$ coincides with one of the following
matrices, where $a$ is a non-negative integer:
\begin{displaymath}
\begin{pmatrix}3 & a\\0 & 1\end{pmatrix},\quad
\begin{pmatrix}3 & 0\\a & 1\end{pmatrix}\quad\text{ or }\quad
\begin{pmatrix}2 & 1\\1 & 2\end{pmatrix}.
\end{displaymath}
Taking Corollary~\ref{Cor5} and Lemma~\ref{L3} into account, up to swapping $s$ and $t$, we have 
\begin{displaymath}
X=\begin{pmatrix}2 & 1\\1 & 2\end{pmatrix},\quad
\Lparen\theta_s\Rparen=\begin{pmatrix}2 & 1\\0 & 0\end{pmatrix},\quad
\Lparen\theta_t\Rparen=\begin{pmatrix}0 & 0\\1 & 2\end{pmatrix}.
\end{displaymath}
However, in this case a direct computation using the fact that 
$\theta_{sts} = \theta_s\theta_t\theta_s - \theta_s$ shows that $\Lparen\theta_{sts}\Rparen = 0$.
This is a contradiction, because $\Lparen\theta_{sts}\Rparen = 0$ implies that $\theta_{sts}$ 
acts as zero, which implies that $\theta_s$ acts as zero as these two $1$-morphisms are in the same two-sided cell.
This excludes the first case.

In the {\bf second} case, that is when $\cos(\frac{2k\pi}{n}) = 0$, we get $k = \frac{n}{4}$
and thus $n \in 4\Z$. In this case, 
$\chi_{\frac{n}{4}}(x) = x^2 - 4x + 2 = (x - 2 - \sqrt{2})(x - 2 + \sqrt{2})$ and thus 
$X$ has eigenvalues $2 \pm \sqrt{2}$, trace $4$ and the determinant $2$. 
This means that, up to reordering of the basis elements, $X$ coincides with one of the following
matrices:
\[
  \begin{pmatrix}
      3 & 1\\
      1 & 1
     \end{pmatrix} \text{ or } 
  \begin{pmatrix}
      2 & 2\\
      1 & 2
     \end{pmatrix}.
\]
Taking Corollary~\ref{Cor5} and Lemma~\ref{L3} into account, up to swapping $s$ and $t$, we have 
\begin{equation}\label{possx2}
X=\begin{pmatrix}2 & 2\\1 & 2\end{pmatrix},\quad
\Lparen\theta_s\Rparen=\begin{pmatrix}2 & 2\\0 & 0\end{pmatrix},\quad
\Lparen\theta_t\Rparen=\begin{pmatrix}0 & 0\\1 & 2\end{pmatrix}.
\end{equation}

In the {\bf third} case, that is when $\cos(\frac{2k\pi}{n}) = \frac{1}{2}$, we get 
$k = \frac{n}{6}$ and thus $n \in 6\Z$. In this case, 
$\chi_{\frac{n}{6}}(x) = x^2 - 4x + 1 = (x - 2 + \sqrt{3})(x - 2 - \sqrt{3})$. 
Therefore $X$ has the eigenvalues $2 \pm \sqrt{3}$, trace $4$ and determinant $1$.
This means that, up to reordering of the basis elements, $X$ coincides with one of the following
matrices:
\[
 \begin{pmatrix}
      3 & 2\\
      1 & 1
     \end{pmatrix},\quad \begin{pmatrix}
      3 & 1\\
      2 & 1
     \end{pmatrix} \text{ or } 
  \begin{pmatrix}
      2 & 3\\
      1 & 2
     \end{pmatrix}.
\]
Taking Corollary~\ref{Cor5} and Lemma~\ref{L3} into account, up to swapping $s$ and $t$, we have 
\begin{equation}\label{possx3}
X=\begin{pmatrix}2 & 3\\1 & 2\end{pmatrix},\quad
\Lparen\theta_s\Rparen=\begin{pmatrix}2 & 3\\0 & 0\end{pmatrix},\quad
\Lparen\theta_t\Rparen=\begin{pmatrix}0 & 0\\1 & 2\end{pmatrix}.
\end{equation}
 
The following arguments which exclude the  cases given by \eqref{possx2} and \eqref{possx3}
follow closely the proof of \cite[Proposition 22]{MazorchukMiemietz5}.

Assume that we are in the situation  given by \eqref{possx2} and consider $\overline{\mathbf{M}}$.
Now, denote by $L_1$ and $L_2$ the simple objects in $\overline{\mathbf{M}}(\c)$ and their indecomposable projective
covers by $P_1$ and $P_2$, respectively. Recall that, since $\theta_s$ is self-adjoint, we have 
$\llbracket \theta_s \rrbracket^t = \Lparen \theta_s \Rparen$ and thus $\theta_s L_2 = 0$. Similarly, we 
see that $\theta_t L_1 = 0$.

Consider $\theta_s L_1$. It has length $3$, with two copies of $L_1$ and one copy of 
$L_2$ in its Jordan-H\"{o}lder filtration. However, by adjunction we have that, 
for any simple module $L$, the module $\theta_w L$ can only have a simple $L'$ in its 
top or socle if $\theta_{w^{-1}} L' \neq 0$. Thus $\theta_s L_1$ cannot be semisimple 
and must have $L_1$ as top and socle.  Therefore it is a uniserial module of Loewy 
length three with top and socle isomorphic to $L_1$. Similarly, the module 
$\theta_t L_2$ has top and socle isomorphic to $L_2$.

Denote by $M$ the homology of the middle term of the complex
\[
 0 \rightarrow L_2 \rightarrow \theta_t\, L_2 \rightarrow L_2 \rightarrow 0.
\]
Then $M$ has length two with both simple subquotients isomorphic to $L_1$. This leaves us with two cases for $M$: 
either $M \simeq L_1 \oplus L_1$ or $M$ is indecomposable. 

Let us now show that $\text{dim}\:\text{Ext}^1(L_1,L_2) = 1$ and that the unique 
non-split extension in this space is obtained as a  quotient of $\theta_s\, L_1$. 
For this, let $N$ be a non-split extension of length two with top $L_1$ and socle $L_2$. 
Then, since $\theta_s L_2 = 0$  and $\theta_s$ is exact, adjunction gives
\[
 \text{dim}\: \text{Hom}(\theta_s\, L_1, N) = \text{dim}\: \text{Hom}(L_1, \theta_s\, N) = \text{dim}\:\text{Hom}(L_1, \theta_s\, L_1) = 1.  
\]
This yields that $N$ is a quotient of $\theta_s L_1$ which implies that $\text{dim}\:\text{Ext}^1(L_1,L_2) = 1$.

Now, if $M\simeq L_1 \oplus L_1$, then $M$ is in the socle of $\theta_t L_2/\mathrm{Soc}(\theta_t L_2)$.
Since $\theta_t L_2$ has simple socle $L_2$, this implies $\text{dim}\:\text{Ext}^1(L_1,L_2) \geq 2$, 
a contradiction. Therefore $M$ is indecomposable.

As $M$ is indecomposable, then,  by adjunction, we get
\[
 1 = \text{dim}\: \text{Hom}(M, \theta_s L_1) = \text{dim}\: \text{Hom}(\theta_s M, L_1).
\]
This implies that $\theta_s M$ has simple top isomorphic to $L_1$ since only $L_1$ 
can occur in the top of $M$, in particular, $\theta_s M$ is indecomposable. Note that $\theta_t L_1 = 0$ 
implies that $\theta_t\theta_s L_1 = \theta_t L_2$ and similarly $\theta_s L_2 = 0$ yields $\theta_s\theta_t L_2 = \theta_s M$. Therefore $\theta_s \theta_t \theta_s L_1 \simeq \theta_s M$ is indecomposable. 
However, $\theta_s \theta_t \theta_s = \theta_{sts} \oplus \theta_s$ and thus $\theta_s M$ must have a direct summand
isomorphic to $\theta_s L_1$. As $\theta_s L_1 \neq 0$ and the Jordan H\"{o}lder multiplicities of $\theta_s L_1$ and $\theta_s M$ are different, this cannot be the case. 
We can conclude that the pair of matrices chosen in \eqref{possx2} can be excluded.

The case \eqref{possx3} is excluded similarly to the case \eqref{possx2}. 
Indeed, using the form of $\llbracket \theta_s \rrbracket$,
we get that $\theta_s L_1$ is uniserial of Loewy length three with simple top and socle isomorphic to $L_1$.
As above, this implies $\dim\mathrm{Ext}^1(L_1,L_2)=\dim\mathrm{Ext}^1(L_2,L_1)=1$. Similarly to the
above, this means that the module $M=\mathrm{Rad}(\theta_t L_2)/\mathrm{Soc}(\theta_t L_2)$ has simple
top and hence $\theta_s\theta_t\theta_s L_1$ is indecomposable and not isomorphic to 
$\theta_s L_1$ by comparing dimensions. This is a contradiction which completes the proof.
\end{proof}

The following theorem is a slight variation of \cite[Theorem~18]{MazorchukMiemietz5} and our 
proof is similar to the one in \cite{MazorchukMiemietz5}. 

\begin{theorem}\label{apexTheorem}
Let $\mathbf{M}$ be a simple transitive $2$-representation of a fiat $2$-category $\mathscr{C}$.
Assume that there is a unique maximal $2$-sided cell $\mathcal{J}$ which does not annihilate $\mathbf{M}$. 
Furthermore, assume that $\mathcal{J}$ is strongly regular. Then $\mathbf{M}$ is equivalent to a cell 
$2$-representation.
\end{theorem}

\begin{proof}
Denote by $\mathscr{C}_{\mathcal{J}}$ the $2$-full $2$-subcategory of $\mathscr{C}$ formed by all $1$-morphisms in 
$\mathcal{J}$ and their respective identity $1$-morphisms. If we simply restrict $\mathbf{M}$ to $\mathscr{C}_{\mathcal{J}}$,
we obtain again a $2$-representation, denoted by $\mathbf{M}_{\mathcal{J}}$. Moreover, $\mathbf{M}_{\mathcal{J}}$ is 
a transitive $2$-representation since the additive closure of $1$-morphisms in $\mathcal{J}$ is stable with 
respect to left multiplication by $1$-morphism in $\mathscr{C}$. 

Next, we claim that $\mathbf{M}_{\mathcal{J}}$ is simple transitive. Assume that $\mathbf{J}$ is an 
ideal of $\mathbf{M}$ which is nonzero and stable with respect to the action of $\mathscr{C}_{\mathcal{J}}$.
Let $\alpha:X \to Y$ be a nonzero morphism in $\mathbf{J}$. By simple transitivity of $\mathbf{M}$, 
there exists a $1$-morphism $G$ in $\mathscr{C}$ such that $G(\alpha): G(X) \to G(Y)$ has an invertible 
nonzero direct summand. Now, composition with $1$-morphisms from $\mathscr{C}_{\mathcal{J}}$ sends this 
direct summand to another invertible morphism. Note that transitivity of $\mathbf{M}_{\mathcal{J}}$ implies 
that there exists $F \in \mathscr{C}_{\mathcal{J}}$ which does not annihilate this summand. However, since $F$ 
is in $\mathcal{J}$, so is $F \circ G$ implying that $F \circ G (\alpha)$ is in $\mathbf{J}$. This proves that 
every $\mathscr{C}_{\mathcal{J}}$-stable ideal contains the identity $\text{id}_X$ for some nonzero object 
$X$ and hence the maximal ideal of $\mathbf{M}_{\mathcal{J}}$ which does not contain any identity $\text{id}_Y$  
is zero. Thus $\mathbf{M}_{\mathcal{J}}$ is simple transitive. 

By \cite[Theorem 18]{MazorchukMiemietz5}, there exists a left cell $\mathcal{L}$ in $\mathcal{J}$ such 
that $\mathbf{M}_{\mathcal{J}}$ is equivalent to the cell $2$-representation $\mathbf{C}^{\mathcal{J}}_{\mathcal{L}}$. 
Moreover, \cite[Theorem 43]{MazorchukMiemietz1} yields that any choice of a left cell $\mathcal{L}$ in $\mathcal{J}$ 
gives us the same $2$-representation, up to equivalence. As usual, there exists $\mathtt{i} = \mathtt{i}_{\mathcal{L}}$ 
in $\mathscr{C}$ such that all $1$-morphisms in $\mathcal{L}$ have domain $\mathtt{i}$. Let $L$ be a simple object in 
$\overline{\mathbf{C}^{\mathcal{J}}_{\mathcal{L}}}$ such that $L$ is not annihilated by any $1$-morphism in $\mathcal{L}$, 
i.e. the simple object corresponding to the Duflo involution of $\mathcal{L}$. Then we can consider $L$ as an object in 
$\overline{\mathbf{M}}(\mathtt{i})$.

Let $\Phi$ be the $2$-natural transformation from the principal $2$-representation $\mathbf{P}_{\mathtt{i}}$ of $\mathscr{C}$ 
to $\overline{\mathbf{M}}$ which sends $P_{\mathbbm{1}_{\mathtt{i}}}$ to $L$. Denote by $\mathbf{N}(\mathtt{j})$ the 
additive closure of all $1$-morphisms $F \in \mathscr{C}(\mathtt{i,j})$ such that $F \geq_L \mathcal{L}$ for 
$\mathtt{j} \in \mathscr{C}$. The image of $\mathbf{N}(\mathtt{j})$ under $\Phi$ is inside the category of projective 
objects in $\overline{\mathbf{M}}(\mathtt{j})$ and contains at least one representative of each isomorphism class of 
indecomposable objects due to results from \cite[Subsection 4.5]{MazorchukMiemietz1}. By construction of $L$, we have that 
the maximal ideal $\mathbf{I}$ in $\mathbf{N}$, not containing $\text{id}_F$ for any $F \in \mathcal{L}$, annihilates $L$. 
Thus, the $2$-representation $\mathbf{K} = \mathbf{N}/\mathbf{I}$ on projective objects in the categories $\mathbf{M}(\mathtt{j})$ 
with $\mathtt{j}\in \mathscr{C}$, is equivalent to the cell $2$-representation $\mathbf{C}_{\mathcal{L}}$ of $\mathscr{C}$. 
By \cite[Theorem 11]{MazorchukMiemietz2}, we get that $\mathbf{K}$ is equivalent to $\mathbf{M}$ and hence $\mathbf{M}$ is 
equivalent to $\mathbf{C}_{\mathcal{L}}$. This completes the proof. 
\end{proof}

\section{The $D_4$-case}\label{sd4}
\label{D4Case}

\subsection{The main result}

In this section we study $\mathscr{S}_4$ which is of special interest as $D_4$ is the Weyl group 
of type  $B_2$ and hence there is a connection to Lie theory, namely,  
$\mathscr{S}_4$ is biequivalent to the $2$-category of projective functors on the
principal block of the BGG category $\mathcal{O}$ for a Lie algebra of type $B_2$, see 
\cite[Subsection~7.2]{MazorchukMiemietz1} for details.  More concretely, we want to 
prove the following theorem which is an analogue of \cite[Theorem~18]{MazorchukMiemietz5} 
where the same result was proved for type $A_n$.

\begin{theorem}\label{MainThm}
Any simple transitive $2$-representation $\mathbf{M}$ of $\mathscr{S}_4$ is equivalent to a cell $2$-representation. 
\end{theorem}

The proof of this is subdivided into several steps. First, we study the possible eigenvalues of 
the matrix $Q := \Lparen\theta_s \oplus \theta_t\Rparen$. Using this we describe 
the decategorifications of $\mathbf{M}$ and show that it is isomorphic to 
the decategorification of a cell $2$-representation. Finally, we explicitly construct 
an equivalence between $\mathbf{M}$ and the corresponding cell $2$-representation.

\subsection{Possible eigenvalues of $Q$}

\begin{lemma}\label{charPoly}
Let $\mathbf{M}$ be a $2$-representation of $\mathscr{S}_4$ such that $\theta_{w_0}\:\mathbf{M}(\c) = 0$. 
Then the polynomial $p(x) =  x^4 - 6x^3 + 10x^2-4x$ annihilates $Q$.
\end{lemma}

\begin{proof}
By Theorem~\ref{Soergel}, we know that $Q$ 
acts on $[\mathscr{S}]^{\C}$ as the element $a := e + s + e + t$ acts on $\C[D_n]$. 
Thus $Q$ is annihilated by the product of all characteristic polynomials 
$\chi_{n,k}$ given in $(\ref{charPolyGroup})$. A simple calculation shows that 
\begin{displaymath}
x(x-2)(x^2-4x+2-\text{cos}(\pi/2))  = x(x-2)(x^2-4x+2) 
= x^4-6x^3+10x^2-4x. 
\end{displaymath}
It remains to note that the only simple $D_4$-modules annihilated by $\underline{w_0}$
and the corresponding characteristic polynomials
for the element $a$ are:
\begin{displaymath}
\begin{array}{c||c|c|c|c}
\text{module}:& V_{-1,-1}&V_{-1,1}&V_{1,-1}&V^{(4,1)} \\
\hline
\chi_a:& x& x-2& x-2&x^2-4x+2-\text{cos}(\pi/2).
\end{array}
\end{displaymath}
As all these polynomials divide $x^4-6x^3+10x^2-4x$, the claim follows.
\end{proof}

\begin{corollary}
Let $\mathbf{M}$ be a $2$-representation of $\mathscr{S}_4$ such that $\theta_{w_0}\mathbf{M}(\c) = 0$. 
Then the possible eigenvalues of $Q$ on $[\mathbf{M}]^{\C}$ are $0, 2-\sqrt{2}, 2, 2 + \sqrt{2}$.
\end{corollary}

\subsection{The structure of $Q$}\label{sub62}
We can now refine the results obtained in Section \ref{sStructofMat} for the case of $\mathscr{S}_4$. 
For this we will use the same notation as in that section which we are going to recall here.

From now on let $\mathbf{M}$ be a simple transitive $2$-representation of $\mathscr{S}_4$. 
As we have seen in Section \ref{sStructofMat} we only need to consider the case where 
$\theta_{w_0}\:\mathbf{M}(\c) = 0$ and $(\theta_s \oplus \theta_t)\:\mathbf{M}(\c) \neq 0$.

Let $X_1, \dots, X_n$ be a complete and irredundant list of pairwise non-isomorphic 
indecomposable objects in $\mathbf{M}(\clubsuit)$. We will in some of the proofs consider 
the abelianization $\overline{\mathbf{M}}$ of $\mathbf{M}$ and then denote the   
indecomposable projective object $0 \to X_i$ in $\overline{\mathbf{M}}(\c)$ by $P_i$. 
Moreover, for $i=1,2,\dots,n$, we denote by $L_i\in \overline{\mathbf{M}}(\c)$ the simple top of $P_i$.

\begin{lemma}\label{L6}
The $[\mathscr{S}_4(\c,\c)]^{\C}$-module $V := [\mathbf{M}]^{\C}$ has the following properties:
\begin{enumerate}[$($a$)$]
\item\label{L6.1} The only possible simple subquotients of $V$ are $V^{(4,1)}$, $V_{-1,1}$ and $V_{1,-1}$.
\item\label{L6.2} The subquotient $V^{(4,1)}$ has multiplicity one in $V$.
\end{enumerate}
\end{lemma}

\begin{proof}
The trivial  $D_4$-module $V_{1,1}$ cannot appear as a subquotient of $V$ since $\underline{w_0}$
does not annihilate $V_{1,1}$, while $\theta_{w_0}$ annihilates $\mathbf{M}(\c)$ by our assumptions.

We have $\dim(V)=n$ by construction. By Corollary~\ref{Cor5}, the trace of $Q$ on
$V$ is $2n$. At the same time, the trace of $Q$ on non-trivial simple $D_4$-modules
is given by:
\begin{displaymath}
\begin{array}{c||c|c|c|c}
\text{module}:&V_{-1,-1}&V_{-1,1}&V_{1,-1}&V^{(4,1)}\\
\hline
\text{dimension}:&1&1&1&2\\
\hline
\text{trace}:&0&2&2&4
\end{array}
\end{displaymath}
Since for the three latest modules the trace equals twice the dimension while for $V_{-1,-1}$
this is not the case, the module $V_{-1,-1}$ cannot appear as a direct summand of $V$. This proves
claim~\eqref{L6.1}.

From claim~\eqref{L6.1} it follows that the only possible eigenvalues for  
$Q$ on $V$ are $2-\sqrt{2}$, $2$ and $2+\sqrt{2}$. If $V^{(4,1)}$
is not a direct summand of $V$, then $Q$ is a diagonal matrix.
As $\mathbf{M}$ is transitive, some power of this matrix must have only positive entries.
Therefore in this case $n=1$ and thus $\mathbf{M}$ decategorifies to either
$V_{-1,1}$ or $V_{1,-1}$ by the previous paragraph. This, however, contradicts Theorem~\ref{noSimple}
which shows that $V^{(4,1)}$ appears in $V$ with nonzero multiplicity.

At the same time, the fact 
that $V^{(4,1)}$ appears in $V$ with nonzero multiplicity 
implies that the maximal eigenvalue of $Q$ on $V$ is $2+\sqrt{2}$.
As $Q$ is non-negative and irreducible 
(since $\mathbf{M}$ is transitive), this eigenvalue must have multiplicity one by
Theorem~\ref{PerronFrobenius}. This proves claim~\eqref{L6.2} and completes the proof.
\end{proof}

\begin{corollary}\label{Cor8+9}
Recall that $n$ denotes the rank of $\mathbf{M}$. We have:
\begin{enumerate}[$($a$)$]
\item\label{Cor8+9.1} $\text{\emph{det}}(Q) = 2^{n-1}$;
\item\label{Cor8+9.2} $\text{\emph{rank}}(Q - 2I_n) = 2$.
\end{enumerate}
\end{corollary}

\begin{proof}
This follows immediately from Lemma~\ref{L6}
since there we established that $2+\sqrt{2}$ and $2 - \sqrt{2}$ are eigenvalues with multiplicity $1$ and 
thus $2$ is an eigenvalue with multiplicity $n-2$.
\end{proof}

\begin{corollary}\label{Cor10}
There exists an ordering of $X_1, \dots, X_n$ such that 
\begin{equation*}
Q = \left(
\begin{array}{c|c}
 2I_k & B \\ \hline
 B' & 2I_{n-k}
\end{array}
\right)
\end{equation*}
where $B$ and $B'$ are positive matrices of rank $1$. 
\end{corollary}

\begin{proof}
From Corollary \ref{Cor5} we get that there exists an ordering of $X_1, \dots, X_n$ and non-negative
matrices $B, B'$ such that $Q$ has the form as described above. What is left to prove is
that $B$ and $B'$ are positive and have rank $1$.

By Corollary~\ref{Cor8+9}, the rank of $Q - 2I_n$ is two. 
Thus the rank of 
\begin{equation*}
Q - 2I_n = \left(
\begin{array}{c|c}
 0 & B \\ \hline
 B' & 0
\end{array}
\right)
\end{equation*}
is two. This leaves us with three options, either the rank of both $B$ and $B'$ is one, 
as we claimed, or the rank of one of them is two and of the other one is zero. 
If we assume that $B$ has rank $0$, then no power of $Q$
can be a positive matrix, which contradicts transitivity of $\mathbf{M}$.
Similarly, $B'$ must be non-zero. Therefore both $B$ and $B'$ have rank one.

If $B$ would have a zero entry, both the column and the row of this entry must
be zero since $B$ has rank one. Assume that this zero column of $B$ is in
the $i$-th column of the matrix $Q$. 
But this means that $G_{Q}$ is not strongly 
connected since there is no directed path from any vertex to $i$.
This again contradicts our assumption that $\mathbf{M}$ is transitive and completes the proof.
\end{proof}

\subsection{$\mathbf{M}$ has rank three}

From Lemma \ref{L6} and Corollary \ref{Cor8+9} we get that $\mathbf{M}$ has rank at least $3$. 
Our goal in this subsection is to prove that $\mathbf{M}$ has rank $3$.

\begin{proposition} \label{prop11}
Let $\mathbf{M}$ be a simple transitive $2$-representation annihilated by $\theta_{w_0}$
and not annihilated by $\theta_s\oplus\theta_t$. Then $\mathbf{M}$ has rank three and 
there exists an ordering of $X_1, X_2, X_3$ such that 
 \begin{equation*}
 Q = \Lparen\theta_s \oplus \theta_t\Rparen = 
 \begin{pmatrix}
  2 & 0 & 1 \\
  0 & 2 & 1 \\
  1 & 1 & 2
 \end{pmatrix}.
\end{equation*}
\end{proposition}
\begin{proof}
 From Corollary \ref{Cor10} we know that 
 \begin{equation*}
 Q = \left(
    \begin{array}{rcl|rcl}
     2 & \cdots & 0 & \lambda_1v_1 & \cdots &\lambda_1 v_l \\
     \vdots & \ddots & \vdots & \vdots &  & \vdots \\
    0 & \cdots & 2 &\lambda_k v_1 & \cdots & \lambda_k v_l \\\hline
    \mu_1w_1 & \cdots & \mu_1w_k & 2 & \cdots & 0 \\
    \vdots & & \vdots & \vdots &\ddots & \vdots \\
    \mu_l w_1 & \cdots & \mu_l w_k & 0 & \cdots & 2 
    \end{array} 
\right),
\end{equation*}
where all $v_i, \mu_i, w_j, \lambda_j$ are positive integers and $n = k + l$. 
Without loss of generality we assume that $\lambda_1 = \mu_1 = 1$. Moreover, we know that 
$\text{det}(Q) = 2^{n-1}$ from Corollary \ref{Cor8+9}. On the other hand, we can calculate 
the determinant of $Q$ as follows: First we subtract suitable multiplies 
of the first row from rows $2,3,\dots,k$  and then  suitable multiples of column 
$k+1$ from columns $k+2$, $k+3$ to get:
\begin{equation*}
\text{det}(Q)  = \left|
     \begin{array}{rcccl|rcccl}
     2 & 0 & \cdots & 0 & 0 & v_1 & 0 & \cdots & 0 & 0 \\
    -2\lambda_2 & \ddots & & & 0 & 0 & & & & 0 \\ 
     \vdots & & \ddots & & \vdots & \vdots & & & & \vdots \\
    -2\lambda_{k-1} & & & \ddots & 0 & 0 & & & & 0 \\ 
    -2\lambda_k & 0 &  \cdots & 0 & 2 & 0 & 0 & \cdots & 0 & 0 \\\hline
    w_1 & w_2 & \cdots & w_{k-1} & w_k & 2 & -2\frac{v_2}{v_1} & \cdots & & -2\frac{v_l}{v_1} \\
    \vdots & & & & \vdots & 0 & \ddots & & & 0  \\
    \vdots & & & & \vdots & \vdots & &\ddots & & \vdots \\
    \vdots & & & & \vdots & 0 & & & \ddots & 0\\
    \mu_lw_1 & \mu_lw_2 & \cdots & \mu_lw_{k-1} & \mu_lw_k & 0 & 0 & \cdots & 0 & 2 
    \end{array} 
\right|
\end{equation*} 

Now we can add suitable multiples of rows $k+2,k+3,\dots$ to row $k+1$ and obtain:
\begin{equation*}
\text{det}(Q)   = \left|
     \begin{array}{rcccl|rcccl}
     2 & 0 & \cdots & 0 & 0 & v_1 & 0 & \cdots & 0 & 0 \\
    -2\lambda_2 & \ddots & & & 0 & 0 & & & & 0 \\ 
     \vdots & & \ddots & & \vdots & \vdots & & & & \vdots \\
    -2\lambda_{k-1} & & & \ddots & 0 & 0 & & & & 0 \\ 
    -2\lambda_k & 0 &  \cdots & 0 & 2 & 0 & 0 & \cdots & 0 & 0 \\\hline
    \tilde{w}_1 & \tilde{w}_2 & \cdots & \tilde{w}_{k-1} & \tilde{w}_k & 2 & 0 & \cdots & & 0 \\
    \vdots & & & & \vdots & 0 & \ddots & & & 0  \\
    \vdots & & & & \vdots & \vdots & &\ddots & & \vdots \\
    \vdots & & & & \vdots & 0 & & & \ddots & 0\\
    \mu_lw_1 & \mu_lw_2 & \cdots & \mu_lw_{k-1} & \mu_lw_k & 0 & 0 & \cdots & 0 & 2 
    \end{array} 
\right|,
\end{equation*}
where $\tilde{w}_i = w_i + \sum_{j=2}^l \mu_jw_i\frac{v_j}{v_1}$.

Now we see that 
\begin{displaymath}
\text{det}(Q)=2^{l-1} \left|
     \begin{array}{rcccl|r}
     2 & 0 & \cdots & 0 & 0 & v_1 \\
    -2\lambda_2 & \ddots & & & 0 & 0 \\ 
     \vdots & & \ddots & & \vdots & \vdots \\
    -2\lambda_{k-1} & & & \ddots & 0 & 0 \\ 
    -2\lambda_k & 0 &  \cdots & 0 & 2 & 0 \\\hline
    \tilde{w}_1 & \tilde{w}_2 & \cdots & \tilde{w}_{k-1} & \tilde{w}_k & 2 \\
    \end{array} 
\right| .
\end{displaymath}
We now expand the determinant with respect to the last column and get:
\begin{displaymath}
 2^{l-1}(-1)^{k+2}v_1 \underbrace{\left|
     \begin{array}{rcccl}
    -2\lambda_2 & 2 & & & 0 \\ 
     \vdots & & \ddots & & \vdots  \\
    -2\lambda_{k-1} & & & \ddots & 0 \\ 
    -2\lambda_k & 0 &  \cdots & 0 & 2 \\
    \tilde{w}_1 & \tilde{w}_2 & \cdots & \tilde{w}_{k-1} & \tilde{w}_k \\
    \end{array} 
\right|}_{=:C_k} + 
2^l \left|
     \begin{array}{rcccl}
     2 & 0 & \cdots & 0 & 0  \\
    -2\lambda_2 & \ddots & & & 0  \\ 
     \vdots & & \ddots & & \vdots  \\
    -2\lambda_{k-1} & & & \ddots & 0  \\ 
    -2\lambda_k & 0 &  \cdots & 0 & 2 \\
    \end{array} 
\right|
\end{displaymath}
The second determinant equals $2^k$. By adding suitable multiples of all rows 
to the last row, we have $C_k=2^{k-1}(-1)^{k-1}\left(\tilde{w}_1 +
\lambda_2\tilde{w}_2+\lambda_3\tilde{w}_3+\dots+\lambda_k\tilde{w}_k\right)$.
This implies $\text{det}(Q)= 2^n-2^{n-2}\left(\sum_{i=1}^k\lambda_iw_i\right)$.

Comparing our two expressions for the determinant of $Q$, we get 
\[
 4-\sum_{j=1}^k\lambda_jw_j\sum_{i=1}^l\mu_iv_i = 2,
\quad
\text{ that is }
\quad
\sum_{j=1}^k\lambda_jw_j\sum_{i=1}^l\mu_iv_i = 2.
\]
We know that $k+l = n \geq 3$ and $\mu_i, \lambda_j, v_i, w_j > 0$. Moreover, we may assume without loss 
of generality that $k \geq 2$. This yields that $k = 2$, $l = 1$ and, furthermore, 
$v_1 = w_1 = w_2 = \lambda_2 = 1$ which completes the proof.
\end{proof}

\begin{corollary}
\label{Cor12}
 There exists an ordering of $X_1, X_2, X_3$ such that
 \begin{align*}
  \Lparen\theta_s\Rparen = \begin{pmatrix}
   2 & 0 & 1 \\
   0 & 2 & 1 \\
   0 & 0 & 0 
  \end{pmatrix},
  \quad
  \Lparen\theta_t\Rparen = \begin{pmatrix}
   0 & 0 & 0 \\
   0 & 0 & 0 \\
   1 & 1 & 2
  \end{pmatrix}\qquad \text{or vice versa.}
 \end{align*}
\end{corollary}

In the following we will by $\Lparen\theta_s\Rparen, \Lparen\theta_t\Rparen$ always refer to the first case 
of the corollary. The second case is similar, by symmetry. 

\begin{remark}\label{absdef}
For the corresponding ordering of $L_1, L_2, L_3$ we have 
 \begin{align*}
  \llbracket\theta_s\rrbracket = \begin{pmatrix}
   2 & 0 & 0 \\
   0 & 2 & 0 \\
   1 & 1 & 0 
  \end{pmatrix},
  \quad
  \llbracket\theta_t\rrbracket = \begin{pmatrix}
   0 & 0 & 1 \\
   0 & 0 & 1 \\
   0 & 0 & 2
  \end{pmatrix}.
 \end{align*}
\end{remark}

\begin{corollary}
The decategorification of $\mathbf{M}$ is isomorphic to the 
decategorification of a cell $2$-representation.
\end{corollary}

\begin{proof}
This follows directly from the description of the decategorifications of the 
cell $2$-representations of $\mathscr{S}_4$ given in Subsection~\ref{DecatS4}.
\end{proof}

\subsection{$\mathbf{M}$ is equivalent to a cell $2$-representation}

We need the following lemma:

\begin{lemma}\label{projective}
Assume the matrices of $\Lparen\theta_s\Rparen$ and $\Lparen\theta_t\Rparen$ are given as in Corollary $\ref{Cor12}$ 
and let $L_1, L_2, L_3$ be the $3$ simple indecomposable objects in $\overline{\mathbf{M}}(\c)$ 
corresponding to $P_1, P_2, P_3$. Then $\theta_s\: L_1$, $\theta_t \theta_s\: L_1$ and 
$\theta_s \theta_t \theta_s\: L_1$ are projective objects.
\end{lemma}

\begin{proof}
Let $A_{\bullet}$, $B_{\bullet}$ and $C_{\bullet}$ be minimal projective resolutions of 
$\theta_s\:L_1$, $\theta_t\theta_s \: L_1$ and $\theta_s\theta_t\theta_s \: L_1$, respectively.
This means that we have an exact sequence
\begin{align*}
\cdots \to A_1 \xrightarrow{\alpha} A_0 \to \theta_s\:L_1\to 0, 
\end{align*}
and similarly for $B_{\bullet}$ (with the corresponding morphism $\beta:B_1\to B_0$) and 
$C_{\bullet}$ (with the corresponding morphism $\gamma:C_1\to C_0$). 
Note that, since $\theta_s \circ \theta_s = \theta_s \oplus \theta_s$, 
we have that $A_{\bullet} \oplus A_{\bullet}$ is a minimal projective resolution of 
$\theta_s\circ\theta_s \: L_1$. 

Consider $\theta_sA_{\bullet}$. This is obviously a projective resolution of $\theta_s\circ \theta_s\:L_1$ since 
$\theta_s$ is exact and sends projectives to projectives (as a self-adjoint functor). 
In particular, $\theta_sA_{\bullet}$ has $A_{\bullet}$ as a direct summand.
The third row in the matrix of  $[\theta_s]$ is zero. This means that $P_3$ does not appear as a direct summand 
of any $\theta_s\:R$ with $R$ projective. This, in turn, implies that $P_3$ cannot appear as a summand
of any $\theta_s \: A_i$. Because of how the first two columns of $[\theta_s]$ look like, we get that $\theta_s$ simply doubles
each direct summand of each  $A_i$. This implies that $\theta_s \: A_{\bullet}$ is a minimal projective resolution
of $\theta_s\circ \theta_s\:L_1\simeq \theta_s\:L_1 \oplus \theta_s\:L_1$.
A similar argument shows that $\theta_s \: C_{\bullet}$ is a minimal projective resolution of 
$\theta_s \theta_t \theta_s\: L_1\oplus \theta_s \theta_t \theta_s\: L_1$.

As $\theta_t \circ \theta_t = \theta_t \oplus \theta_t$, we have that $B_{\bullet} \oplus B_{\bullet}$ 
is a minimal projective resolution of $\theta_t \circ (\theta_t \circ\theta_s L_1)$. Similarly to the above,
for $R$ projective, the module $\theta_t\:R$ is a direct sum of copies of $P_3$, moreover, $\theta_t$
doubles the module $P_3$. It follows that  $\theta_t\: B_{\bullet}$ is a minimal projective resolution of 
$\theta_t \circ (\theta_t \circ\theta_s \: L_1)$. 

Next, we notice that we have 
$\theta_t\theta_s\theta_t\theta_s = \theta_{ts}\theta_{ts} = \theta_{ts} \oplus \theta_{ts} \oplus \theta_{tsts}$.
Since, by assumptions, $\theta_{tsts}\:\overline{\mathbf{M}}(\c) = \theta_{w_0}\:\overline{\mathbf{M}}(\c) = 0$, 
the last summand in this decomposition acts as zero. As the matrix of $\theta_t\theta_s$ sends every projective $[P_i]$ to $2[P_3]$,
similarly to the above we obtain  $\theta_t\theta_s\: B_{\bullet} = B_{\bullet}\oplus B_{\bullet}$. 

Now we observe that 
$\theta_s\:B_{\bullet}$ is a projective resolution of $\theta_s\theta_t\theta_s\:L_1$ 
and hence contains $C_{\bullet}$ as a direct summand.
Therefore $\theta_t\theta_s \: B_{\bullet} = B_{\bullet} \oplus B_{\bullet}$ contains 
$\theta_t\:C_{\bullet}$ as a direct summand. On the other hand, $\theta_t \:C_{\bullet}$ contains 
$\theta_t\theta_s\:B_{\bullet} = B_{\bullet} \oplus B_{\bullet}$ as a summand, by a similar argument.
Thus $\theta_t\: C_{\bullet} = B_{\bullet} \oplus B_{\bullet}$ and hence $\theta_t\: C_{\bullet}$ 
is minimal. 

Note that $\theta_t$ does not annihilate any projective objects. Therefore it sends non-minimal resolutions
to non-minimal resolutions. This implies that $\theta_s\:B_{\bullet}$ is minimal. 

Finally, we observe that $\theta_s\theta_t\theta_s = \theta_s \oplus \theta_{sts}$ and therefore 
$A_{\bullet}$ is,  in fact, a direct summand of $C_{\bullet}$. Therefore
$\theta_t\,A_{\bullet}$ is a direct summand of 
$\theta_t\,C_{\bullet}= B_{\bullet} \oplus B_{\bullet}$.
We note that 
$\theta_t\theta_s\, L_1\neq 0$, as follows directly from the matrices $\llbracket\theta_s\rrbracket$ 
and $\llbracket\theta_t\rrbracket$  given in Remark~\ref{absdef}. Therefore $\theta_t\,A_{\bullet}$
is nonzero and hence must be equal to $B_{\bullet}$. This shows that $\theta_t\,A_{\bullet}$ is minimal.

The above shows that both $\theta_s$ and $\theta_t$ map $A_{\bullet}$, $B_{\bullet}$ and $C_{\bullet}$
to minimal resolutions and the latter are, moreover, direct sums of copies of 
$A_{\bullet}$, $B_{\bullet}$ and $C_{\bullet}$.

Denote by $\mathcal{P}$ the category of projective objects in $\overline{\mathbf{M}}$.
This category is equivalent to $\mathbf{M}(\c)$. 
In particular, it carries the structure of a simple transitive 
$2$-representation of $\mathscr{S}_4$ by restriction. Consider the ideal $\mathcal{I}$ of $\mathcal{P}$ generated by $\alpha$, $\beta$ and $\gamma$. 
From the minimality of $A_{\bullet}$, $B_{\bullet}$ and $C_{\bullet}$ it 
follows that $\mathcal{I}$ is contained in the radical of $\mathcal{P}$. Since both $\theta_s$ and $\theta_t$
send $A_{\bullet}$, $B_{\bullet}$ and $C_{\bullet}$
to direct sums of copies of $A_{\bullet}$, $B_{\bullet}$ and $C_{\bullet}$, 
the ideal $\mathcal{I}$
is stable under the action of both $\theta_s$ and $\theta_t$. Thus simple 
transitivity of $\mathbf{M}$ implies that $\mathcal{I} = 0$ and hence $\alpha = \beta=\gamma=0$. 
The claim follows.
\end{proof}  

\begin{theorem}
If $\mathbf{M}$ is a simple transitive $2$-representation of $\mathscr{S}_4$, 
then $\mathbf{M}$ is equivalent to a cell $2$-representation. 
\end{theorem}

\begin{proof}
As we have seen above, we only need to consider the case for which we have $\theta_{w_0} \mathbf{M}(\c)= 0$ and 
$(\theta_s \oplus \theta_t) \mathbf{M}(\c) \neq 0$ since all other cases are already covered, 
see Subsection~\ref{sub62}.  Moreover, from Proposition~\ref{prop11} we know that the rank of 
$\mathbf{M}$ in case $\theta_{w_0} \mathbf{M}(\c)= 0$ and 
$(\theta_s \oplus \theta_t) \mathbf{M}(\c) \neq 0$ is $3$. 

Assume that $\Lparen\theta_s\Rparen$ and $\Lparen\theta_t\Rparen$ are given as in Corollary \ref{Cor12}
(the other case is dealt with in a similar way). We claim that $\mathbf{M}$ is equivalent to 
$\mathbf{C}_{\mathcal{L}_s}$. Recall that $\mathbf{C}_{\mathcal{L}_s}$ is the quotient of 
$\text{add}(\{\theta_s,\theta_{ts},\theta_{sts},\theta_{tsts}\})$ inside the principal $2$-representation
of $\mathscr{S}_4$ by the unique maximal $\mathscr{S}_4$-invariant ideal $\mathbf{I}$ given by 
Lemma~\ref{SimpTransIdeal}. Recall also that $\mathbf{M}$ is equivalent to the $2$-subrepresentation 
$\overline{\mathbf{M}}_{\text{pr}}$ of $\overline{\mathbf{M}}$ given by restricting the action of 
$\mathscr{S}_4$ to the subcategory of projective objects in $\overline{\mathbf{M}}(\c)$. 
 
Due to Remark~\ref{absdef}, the module $\theta_s \: L_1$ has two simple subquotients,
namely $L_1$ with multiplicity two and $L_3$ with multiplicity one. As $\theta_s$ annihilates $L_3$,
by adjunction, $P_3$ cannot be a direct summand of $\theta_s \: L_1$. Since
$\theta_s \: L_1$ is projective by Lemma~\ref{projective}, we thus either have $\theta_s \: L_1=P_1$
or $\theta_s \: L_1=P_1\oplus P_1$. In the latter case all composition multiplicities of 
$\theta_s \: L_1$ should be even. Therefore $\theta_s \: L_1=P_1$. From
Corollary~\ref{Cor12} it follows that  $\theta_t\theta_s\: L_1 = P_3$.
By a similar argument one shows that $\theta_{sts}\:L_1 = P_2$.
 
Consider the canonical $2$-natural transformation 
\begin{displaymath}
\begin{array}{rl}
  \Phi: \mathbf{P}_{\c} &\to \overline{\mathbf{M}}, \\
	\mathbbm{1}_{\c} &\mapsto L_1. \\
\end{array}
\end{displaymath}
This restricts to a map from 
$\mathbf{N}(\c):=\text{add}(\{\theta_s,\theta_{ts},\theta_{sts},\theta_{tsts}\})$ inside $\mathbf{P}_{\c}$ 
to $\overline{\mathbf{M}}_{\text{pr}}(\c)$.
By construction, we have  $\mathbf{C}_{\mathcal{L}_s} \simeq \mathbf{N}/\mathbf{I}$ where $\mathbf{I}$ is the 
unique maximal ideal that does not contain $\text{id}_F$ for any $F \in \mathcal{L}_s$. 
Consider the quotient $\mathbf{M}' = \overline{\mathbf{M}}_{\text{pr}}/\Phi(\mathbf{I})$. 
Since $\overline{\mathbf{M}}_{\text{pr}}$ surjects onto $\mathbf{M'}$, the Cartan matrix 
of $\mathbf{M'}$ is component-wise less that or equal to the Cartan matrix of the 
$\overline{\mathbf{M}}_{\text{pr}}$ (here by the Cartan matrix we mean the matrix which gives dimensions of
homomorphism spaces between indecomposable projective objects). At the same time,
the Cartan matrix 
of $\mathbf{M'}$ is component-wise greater than or equal to the Cartan matrix of 
$\mathbf{C}_{\mathcal{L}_s}$ (since endomorphisms in $\mathbf{C}_{\mathcal{L}_s}(\c)$ 
embed into endomorphisms of  $\mathbf{M'}(\c)$ by construction).

However, all subquotients of indecomposable projective object $\overline{\mathbf{M}}_{\text{pr}}$ 
are determined above and we see that the Cartan matrix for $\overline{\mathbf{M}}_{\text{pr}}$
equals the Cartan matrix for $\mathbf{C}_{\mathcal{L}_s}$. This implies that $\Phi(\mathbf{I})=0$
and $\mathbf{M}'=\overline{\mathbf{M}}_{\text{pr}}$.
This yields that $\Phi$ induces an equivalence from $\mathbf{C}_{\mathcal{L}_s}$
to $\overline{\mathbf{M}}_{\text{pr}}$ and completes the proof.
\end{proof}

\section{Simple transitive $2$-representations of $\mathscr{S}_n$ of small ranks} 
\label{Rank1and2Case}

Here we study simple transitive $2$-representations of rank one or two for all $\mathscr{S}_n$. 

\subsection{The rank one case}

\begin{theorem}
Let $\mathbf{M}$ be a simple transitive $2$-representation of $\mathscr{S}_n$ which has rank one.
Then $\mathbf{M}$ is equivalent to either $\mathbf{C}_{\mathcal{L}_e}$ or $\mathbf{C}_{\mathcal{L}_{w_0}}$. 
\end{theorem}

\begin{proof}
The decategorification of $\mathbf{M}$ is one-dimensional and hence simple. Therefore, by Theorem~\ref{noSimple},
this decategorification is isomorphic to either the sign $D_n$-module $V_{-1,-1}$ or to the
trivial $D_n$-module $V_{1,1}$. 

If the decategorification of $\mathbf{M}$ is isomorphic to $V_{-1,-1}$,
we have $\mathbf{M}(\theta_s\oplus\theta_t)=0$. Therefore $\mathbf{M}$ factors through the 
quotient of $\mathscr{S}_n$ by the $2$-ideal generated by all $\theta_w$, $w\neq e$.
This quotient has a unique non-trivial indecomposable $1$-morphism, namely, the identity, up to 
isomorphism. Therefore it has strongly regular cells and thus a unique simple transitive
$2$-representation, which is exactly $\mathbf{C}_{\mathcal{L}_e}$.

If the decategorification of $\mathbf{M}$ is isomorphic to $V_{1,1}$,
we have $\mathbf{M}(\theta_{stst})\neq 0$. Then the fact that $\mathbf{M}$ is equivalent
to $\mathbf{C}_{\mathcal{L}_{w_0}}$ follows from Theorem~\ref{apexTheorem}.
\end{proof}

\subsection{The rank two case}
Our main result in this subsection is the following:

\begin{theorem}\label{no2dim}
No simple transitive $2$-representation of $\mathscr{S}_n$ has rank two.
\end{theorem}

\begin{proof}
By Theorem~\ref{noSimple}, the decategorification of $\mathbf{M}$ is not simple
as a $\C[D_n]$-module. If $n$ is odd, we have the following three 
cases for this decategorification:
\begin{align*}
V_{1,1} \oplus V_{1,1}, \quad V_{1,1} \oplus V_{-1,-1}, \quad V_{-1,-1} \oplus V_{-1,-1}. 
\end{align*}
The first and the third case cannot occur since the matrix $\Lparen\theta_s\oplus\theta_t\Rparen$
in these cases is diagonal and thus not irreducible, which
contradicts transitivity. In the second case, we have $\Lparen\theta_s\Rparen = \Lparen\theta_t\Rparen$, 
moreover, both matrices have determinant $0$ and trace $2$. Therefore each power of
$\Lparen\theta_s\oplus\theta_t\Rparen$ is a scalar multiple of $\Lparen\theta_s\Rparen$. Similarly to the
proof of Lemma~\ref{L3}, one shows that the diagonal entries of $\Lparen\theta_s\Rparen$ can only
be $0$ or $2$. Since both rank and trace are $2$, $\Lparen\theta_s\Rparen$ contains
a zero diagonal entry. Therefore any power of $\Lparen\theta_s\oplus\theta_t\Rparen$ contains a
zero entry which again contradicts transitivity. 

If $n$ is even, we have the three cases above and they all are excluded by the same arguments
as above. However, we also have the following seven extra cases:
\begin{align*}
& V_{1,1} \oplus V_{1,-1},\quad  V_{1,1} \oplus V_{-1,1},\quad 
V_{1,-1} \oplus V_{1,-1},\quad  V_{1,-1} \oplus V_{-1,1},\\
& V_{1,-1} \oplus V_{-1,-1},\quad  V_{-1,1} \oplus V_{-1,1}, \quad  V_{-1,1} \oplus V_{-1,-1}
\end{align*}
The cases where the same simple is repeated (the third case in the first row and
the second case in the second row) again lead to a diagonal matrix $\Lparen\theta_s\oplus\theta_t\Rparen$ 
and hence are excluded by the same argument as in the previous paragraph. Similarly, the
last case in the first row leads to a diagonal matrix $\Lparen\theta_s\oplus\theta_t\Rparen$ and thus
is excluded.

In the first case of the second row we have $\mathbf{M}(\theta_t)=0$ while $\mathbf{M}(\theta_s)\neq 0$.
This is not possible as $\theta_t$ and $\theta_s$ belong to the same two-sided cell.
A similar argument with $s$ and $t$ swapped excludes the third case in the second row.

This leaves us with the first two cases in the first row. By symmetry, it is enough to show
that we can exclude the first case. In it $\Lparen\theta_s\Rparen$ is twice the identity matrix while
$\Lparen\theta_t\Rparen$ has determinant $0$ and trace $2$. As we mentioned above, the only possibility for
the diagonal entries of $\Lparen\theta_t\Rparen$ are $0$ and $2$. Therefore, up to permutation of basis vectors,
the matrix $\Lparen\theta_t\Rparen$ has one the following two forms:
\begin{displaymath}
\left(\begin{array}{cc}2&0\\x&0\end{array}\right)\quad\text{ or }\quad
\left(\begin{array}{cc}2&x\\0&0\end{array}\right).
\end{displaymath}
In this case the matrix $\Lparen\theta_s\oplus\theta_t\Rparen$
is either upper or lower triangular and hence not irreducible. 
This contradicts transitivity of $\mathbf{M}$ and completes the proof.
\end{proof}
\nopagebreak

{\bf Acknowledgments.}  The author is very grateful to his supervisor 
Volodymyr Mazorchuk who helped with a lot of ideas and patiently answered all  questions. 
Moreover, he wants to thank the referees for pointing out a gap in the original proof of 
Theorem~\ref{noSimple} as well as for many comments which helped to improve the paper greatly.

\end{document}